\documentclass[10pt]{article}

\usepackage[T1]{fontenc}
\usepackage{lmodern}

\usepackage{amsmath, amssymb, amsthm}
\usepackage[margin=2cm]{geometry}
\usepackage{color, graphicx}
\usepackage[breakable]{tcolorbox}
\usepackage{multirow}

\usepackage{etoolbox}
\makeatletter
\patchcmd{\@maketitle}{\LARGE \@title}{\LARGE\bfseries\@title}{}{}

\renewcommand{\@seccntformat}[1]{\csname the#1\endcsname.\quad}
\makeatother

\definecolor{darkblue}{rgb}{0,0,.5}

\usepackage{hyperref}
\hypersetup{
	colorlinks=true,        
	linkcolor=darkblue,         
	citecolor=darkblue,         
	urlcolor=darkblue          
}

\makeatletter
\def\th@plain{%
	\thm@notefont{}
	\itshape 
}
\def\th@definition{%
	\thm@notefont{}
	\normalfont 
}

\renewenvironment{proof}[1][\proofname]{\par
	\normalfont
	\topsep0\p@\@plus3\p@ \trivlist
	\item[\hskip\labelsep\itshape
	#1\@addpunct{.}]\ignorespaces
}{%
	\qed\endtrivlist
}
\makeatother

\newtheorem{theorem}{Theorem}[section]
\newtheorem{lemma}[theorem]{Lemma}
\newtheorem{corollary}[theorem]{Corollary}
\newtheorem{proposition}[theorem]{Proposition}
\theoremstyle{definition}
\newtheorem{definition}[theorem]{Definition}
\theoremstyle{definition}
\newtheorem{example}[theorem]{Example}
\theoremstyle{definition}
\newtheorem{remark}[theorem]{Remark}
\theoremstyle{definition}
\newtheorem{assumption}{Assumption}
\theoremstyle{definition}
\newtheorem{algorithm}{Algorithm}

\renewcommand\theenumi{(\roman{enumi})}

\renewcommand{\labelenumi}{\rm (\roman{enumi})}

\usepackage[shortlabels]{enumitem}

\allowdisplaybreaks

\newcommand{\scal}[2]{\left\langle {#1},{#2} \right\rangle}

\newcommand{\RP}{\ensuremath{\mathbb{R}_+}}
\newcommand{\RPP}{\ensuremath{\mathbb{R}_{++}}}
\newcommand{\st}{\ensuremath{\stackrel}}

\newcommand{\argmin}{\ensuremath{\operatorname*{argmin}}}

\newcommand{\inte}{\ensuremath{\operatorname{int}}}

\newcommand{\dom}{\ensuremath{\operatorname{dom}}}

\newcommand{\epi}{\ensuremath{\operatorname{epi}}}

\newcommand{\dist}{\ensuremath{\operatorname{dist}}}

\newcounter{step}[algorithm]
\newcommand\step[1]{%
	\refstepcounter{step}	
	\vskip 0.25\baselineskip
	\ifx\hfuzz#1\hfuzz
		\item[~\(\triangleright\)~\textbf{Step~\arabic{step}.}]
	\else
		\item[~\(\triangleright\)~\textbf{Step~\arabic{step}}] (\texttt{#1})\textbf{.}%
	\fi
}

\begin{document}

\title{Extrapolated Proximal Subgradient Algorithms for Nonconvex and Nonsmooth Fractional Programs\footnote{Research of the first author is partially supported by the Austrian Science Fund (FWF), project number I 2419-N32, and
the research of the second and the third authors is partially supported by the Australian Research Council (ARC), project number DP190100555.}}

\author{
Radu Ioan Bo\c{t}\thanks{Faculty of Mathematics, University of Vienna, A-1090 Vienna, Austria.
E-mail: \texttt{radu.bot@univie.ac.at}.},
~
Minh N. Dao\thanks{School of Engineering, Information Technology and Physical Sciences, Federation University Australia, Ballarat 3353, Australia.
E-mail: \texttt{m.dao@federation.edu.au}.},
~and~
Guoyin Li\thanks{Department of Applied Mathematics, University of New South Wales, Sydney 2052, Australia.
E-mail: \texttt{g.li@unsw.edu.au}.}
}

\date{Revised Version: Oct. 16, 2020}
\maketitle

\vspace{-0.2cm}
\begin{abstract}
In this paper, we consider a broad class of nonsmooth and nonconvex fractional programs, where the numerator  can be written as the sum of a continuously differentiable convex function whose gradient is Lipschitz continuous and  a proper lower semicontinuous (possibly nonconvex) function, and the denominator is weakly convex over the constraint set. This model problem includes the composite optimization problems studied extensively lately, and encompasses many important modern fractional optimization problems arising from diverse areas such as the recently proposed scale invariant sparse signal reconstruction problem in signal processing. We propose a proximal subgradient algorithm with extrapolations for solving this optimization model and show that the iterated sequence generated by the algorithm is bounded and any of its limit points is a stationary point of the model problem. The choice of our extrapolation parameter is flexible and includes the popular extrapolation parameter adopted in the restarted Fast Iterative Shrinking-Threshold Algorithm (FISTA). By providing a unified analysis framework of descent methods, we establish the convergence of the full sequence under the assumption that a suitable merit function satisfies the Kurdyka--{\L}ojasiewicz (KL) property. In particular, our algorithm exhibits \emph{linear convergence} for the scale invariant sparse signal reconstruction problem and the Rayleigh quotient problem over spherical constraint. In the case where the denominator is the maximum of finitely many continuously differentiable weakly convex functions, we also propose an enhanced extrapolated proximal subgradient algorithm with guaranteed convergence to a stronger notion of stationary points of the model problem. Finally, we illustrate the proposed methods by both analytical and simulated numerical examples.
\end{abstract}

\section{Introduction}

In this paper, we consider the following class of nonsmooth and nonconvex fractional program which takes the form
\begin{equation}\label{e:prob}
\min_{x\in S} \frac{f(x)}{g(x)}, \tag{P}
\end{equation}
where $\mathcal{H}$ is a finite-dimensional real Hilbert space, $S$ is a nonempty closed convex subset of $\mathcal{H}$, and $f, g\colon \mathcal{H}\to \left(-\infty,+\infty\right]$ are proper lower semicontinuous functions which are not necessarily convex.
Throughout this paper, we assume that the numerator $f$ can be written as the sum of $f^\mathfrak{s}$ and $f^\mathfrak{n}$, where $f^\mathfrak{s}$ is a continuously differentiable convex function  whose gradient is Lipschitz continuous and $f^\mathfrak{n}$ is a nonconvex function, and the denominator $g$ is finite, positive, and weakly convex over the constraint set $S$. We note that weakly convex functions form a broad class of functions which covers convex functions, nonconvex quadratic functions and continuously differentiable functions whose gradient are Lipschitz continuous.

This class of nonsmooth and nonconvex fractional program is a broad optimization model which encompasses many important modern optimization problems arising from diverse areas. This includes, for example, the recently proposed scale invariant sparse signal reconstruction problem in signal processing \cite{RWDL19} and the robust Sharpe ratio optimization problems in finance \cite{CHZ11}. Moreover, in the special case where the denominator $g(x) \equiv 1$ and $S=\mathcal{H}$, problem~\eqref{e:prob} reduces to the well-studied nonsmooth composite optimization with the form
\begin{equation*}
 \min_{x\in \mathcal{H}} f(x)=f^\mathfrak{s}(x)+f^\mathfrak{n}(x),
\end{equation*}
which covers a lot of modern optimization problems in machine learning (for example, the Lasso problem in computer science). Below we provide a few motivating examples illustrating the model problem~\eqref{e:prob}.

\begin{enumerate}
\item\label{r:ex_ssr}
{\bf Scale invariant sparse signal recovery problem:} In signal processing, to reconstruct a sparse signal from its observation, one considers the following scale invariant minimization problem \cite{RWDL19}
\begin{equation*}
\min_{x\in \mathbb{R}^N} \frac{\|x\|_1}{\|x\|_2} \quad\text{s.t.}\quad Ax \leq b,\ Cx =d,
\end{equation*}
where $\|\cdot\|_1$ and $\|\cdot\|_2$ are the $\ell_1$-norm and Euclidean norm respectively, $A \in \mathbb{R}^{M\times N}$, $b \in \mathbb{R}^M$, $C \in \mathbb{R}^{P\times N}$, $d \in \mathbb{R}^P$, and $S =\{x\in \mathbb{R}^N: Ax \leq b, Cx =d\}$ is bounded and does not contain the origin. Here, the objective function relates to the restricted isometry constant and serves as a surrogate of the cardinality of $x$. It was shown in \cite{RWDL19} that this model can outperform the celebrated Lasso model in recovering a sparse solution. This model problem is indeed a special case of problem~\eqref{e:prob} with $f(x) =\|x\|_1$, $g(x) =\|x\|_2$ and $S$ being a polytope with the form that $S =\{x\in \mathbb{R}^N: Ax \leq b, Cx =d\}$.

\item\label{r:rayleigh}
{\bf Rayleigh quotient optimization with spherical constraint:}
The Rayleigh quotient optimization problem with spherical constraint can be formulated as
\begin{equation*}
\min_{x\in \mathbb{R}^N} \frac{x^\top Ax}{x^\top Bx} \quad\text{s.t.}\quad \|x\|_2=1,
\end{equation*}
where $A$ and $B$ are symmetric positive definite matrices. This is a special
case of problem~\eqref{e:prob} with $S=\mathbb{R}^N$, $f(x)=x^\top Ax+ \iota_C(x)$ where $C$ is  the unit sphere $\{x\in \mathbb{R}^N: \|x\|_2=1\}$ and $\iota_C$ is the indicator function of the set $C$ (see \eqref{eq:indicator} later for the definition of indicator function), and $g(x)=x^\top Bx$.

\item\label{r:ex_Sharpe}
{\bf Robust Sharpe ratio minimization problem:} The standard Sharpe ratio optimization problem (see, e.g., \cite{CHZ11}) can be formulated as
\begin{equation*}
\max_{x\in \mathbb{R}^N} \frac{a^\top x-r}{\sqrt{x^\top Ax}}
\quad\text{s.t.}\quad e^\top x=1,\ x\geq 0,
\end{equation*}
where the numerator is the expected return and the denominator measures the risk. In practice, the data associated with the model is often uncertain due to prediction or estimation errors. Following
robust optimization approach, we assume that the data $(A,a,r)$ are uncertain and belong to the polyhedral uncertainty set
$\mathcal{U}= \mathcal{U}_1 \times \mathcal{U}_2$,
where
$\mathcal{U}_1={\rm conv}\{(a_1,r_1),\dots,(a_{m_1},r_{m_1})\}$  and  $\mathcal{U}_2={\rm conv}\{A_1,\dots, A_{m_2}\}$.
Here, $(a_i,r_i) \in \mathbb{R}^N \times \mathbb{R}$, $i=1,\dots,m_1$, are such that $a_i^\top x-r_i \leq 0$ for all $x \in S$, and $A_j$ are symmetric positive definite matrix, $j=1,\dots,m_2$.
The robust Sharpe ratio optimization problem can be  written as
\begin{equation*}
\max_{x\in \mathbb{R}^N} \frac{\min_{(a,r) \in \mathcal{U}_2} \{a^\top x-r\}}{\max_{A \in \mathcal{U}_2}\sqrt{x^\top Ax}}
\quad\text{s.t.}\quad e^\top x=1,\ x\geq 0,
\end{equation*}
which can be further simplified as
\begin{equation*}
\min_{x\in \mathbb{R}^N} - \frac{\min_{1 \leq i\leq m_1} \{a_i^\top x-r_i\}}{\max_{1 \leq i \leq m_2}{\sqrt{x^\top A_ix}}}
\quad\text{s.t.}\quad e^\top x=1,\ x\geq 0.
\end{equation*}
This is a special case of problem~\eqref{e:prob} with $f(x)=-\min_{1 \leq i\leq m_1} \{a_i^\top x-r_i\}=\max_{1 \leq i\leq m_1} \{r_i-a_i^\top x\}$, $g(x)=\max_{1 \leq i \leq m_2}\sqrt{x^\top A_ix}$ and $S=\{x\in \mathbb{R}^N: e^\top x=1, x\geq 0\}$.
\end{enumerate}

The fractional programming problem has a long history, and a classical and popular approach for solving the fractional programming problem is the Dinkelbach's method (see, for example, \cite{CFS85,Din67}) which relates it to the following optimization problem
\begin{equation}\label{eq:P2}
\min_{x \in S} f(x) -\overline{\theta} g(x).
\end{equation}
In particular, \eqref{e:prob} has an optimal solution $\overline{x} \in S$ if and only if $\overline{x}$  is  an optimal solution to \eqref{eq:P2} and the
optimal objective value of \eqref{eq:P2} is equal to zero with $\overline{\theta}=\frac{f(\overline{x})}{g(\overline{x})}$. However, one drawback of this procedure is that this can only be done in the very restrictive case when the optimal objective value of \eqref{e:prob} is known. To overcome this drawback, in the literature (see \cite{CFS85,Din67,Iba81,Iba83,Sch75}) an iterative scheme was proposed which  requires solving in
each iteration $n$ of the optimization problem
\begin{equation}\label{eq:P3}
\min_{x \in S} \{f(x) -{\theta}_n g(x)\}
\end{equation}
while $\theta_n$ is updated by $\theta_{n+1} :=
\frac{f(x_{n+1})}{g(x_{n+1})}$, where $x_{n+1}$ is an optimal solution of \eqref{eq:P3}. However,
solving in each iteration an optimization problem of type \eqref{eq:P3} can be as expensive and
difficult as solving the fractional programming problem~\eqref{e:prob} in general.

Recently, in view of the success of the proximal algorithms in solving composite optimization problems (that is, when the denominator $g(x) \equiv 1$), \cite{BotCse17} proposed proximal gradient  type algorithms  for fractional
programming problems, where the numerator $f$ is a proper, convex and lower semicontinuous function and the denominator $g$ is a smooth function, either concave
or convex. The approach of \cite{BotCse17} is appealing because the proposed iterative methods there perform a gradient step with respect to $g$ and a proximal step with respect to $f$. In this way, the functions $f$ and $g$ are processed separately in each iteration.

Although the approach in \cite{BotCse17} is very inspiring, still many research questions need to be answered. For example,
\begin{itemize}
\item
firstly, how to extend the approach in \cite{BotCse17} to the case where the numerator and denominator are both nonconvex and nonsmooth? Such an extension would allow us to cover, for example, robust Sharpe ratio optimization problems where both the numerator $f$ and the denominator $g$ are nonsmooth, and the Rayleigh quotient optimization problem with spherical constraints where the numerator $f$ is a nonconvex function.
\item
secondly, it is known that the performance of the proximal gradient method can be largely improved (see \cite{Nes04}) if one can incorporate extrapolation steps  in solving composite optimization problems (that is, when the denominator $g(x) \equiv 1$ in \eqref{e:prob}), as for example for the restarted Fast Iterative Shrinking-Threshold Algorithm (FISTA) \cite[Chapter~10]{Beck17}. Therefore, it is of great interest to develop proximal algorithms with extrapolations for solving fractional programs.
\item
thirdly, in the case where $f$ and $g$ are convex, and $g$ is continuously differentiable, it was shown in \cite{BotCse17} that the proximal gradient method generates a sequence of iterates which converges to a stationary point of problem~\eqref{e:prob}. Recently, algorithms were proposed for computing a stronger version of stationary points called d(irectional)-stationary points for a class of difference-of-convex optimization problems (for example see \cite{AV20,PRA16}). Taking this into account, developing  algorithms which converge to sharper notions of stationary points of problem~\eqref{e:prob} is also highly desirable.
\end{itemize}

The purpose of this paper is to provide answers to the above questions. Specifically, the contributions of this paper are as follows.
\begin{enumerate}[label =(\arabic*)]
\item
In Section~\ref{s:ePSG}, we propose a proximal subgradient algorithm with extrapolations for solving the model problem~\eqref{e:prob}. We then establish that the sequence of iterates generated by the algorithm is bounded and any of its limit points is a stationary point of the model problem~\eqref{e:prob}. Interestingly, the  convergence of our algorithm does not require the numerator and denominator to be convex or smooth. Moreover, our extrapolation parameter is broad enough to accommodate the popular extrapolation parameter used for restarted FISTA.
\item
In Section~\ref{s:KL}, we establish a general framework for analyzing descent methods which is amenable for optimization methods with multi-steps and inexact subproblems. Our conditions are weaker than those in the literature  and complement the existing results. With the help of this framework, we establish the convergence of the full sequence under the assumption that a suitable merit function satisfies the KL property. In particular, by identifying the explicit KL exponent, we establish linear convergence of the proposed algorithm for scale invariant sparse signal recovery problem and Rayleigh quotient optimization with spherical constraint.
\item
In the case where the denominator is the maximum of finitely many continuously differentiable weakly convex functions, in Section~\ref{s:strong}, we also propose an enhanced proximal subgradient algorithm with extrapolations, and show that this enhanced algorithm converges to a stronger notion of stationary points of the model problem.

\item
Finally, we illustrate the proposed methods via analytical and simulated numerical examples in Section~\ref{s:numerical}.
\end{enumerate}

\section{Preliminaries}
Throughout this work, we assume that $\mathcal{H}$ is a finite-dimensional real Hilbert space
with inner product $\scal{\cdot}{\cdot}$ and the induced norm $\|\cdot\|$. The set of nonnegative integers is denoted by $\mathbb{N}$, the set of real numbers by $\mathbb{R}$, the set of nonnegative real numbers by $\RP := \{x \in \mathbb{R}:x \geq 0\}$, and the set of the positive real numbers by $\RPP := \{x \in \mathbb{R}:x >0\}$.

Let $h\colon \mathcal{H}\to \left[-\infty,+\infty\right]$ be an extended real-valued function. The \emph{domain} of $h$ is $\dom h :=\{x \in \mathcal{H}: h(x) <+\infty\}$. We say that $h$ is \emph{proper} if $\dom h\neq \varnothing$ and it never take the value $-\infty$. The function $h$ is \emph{lower semicontinuous} if, for all $x\in \dom h$, $h(x)\leq \liminf_{z\to x} h(z)$. We use the symbol $z\st{h}{\to}x$ to indicate $z\to x$ and $h(z)\to h(x)$. Given $x\in \mathcal{H}$ with $|h(x)| <+\infty$, the \emph{Fr\'echet subdifferential} of $h$ at $x$ is defined by
\begin{equation*}
\widehat{\partial} h(x) :=\left\{u\in \mathcal{H}:\; \liminf_{z\to x}\frac{h(z)-h(x)-\langle u,z-x\rangle}{\|z-x\|}\geq 0\right\}
\end{equation*}
and the \emph{limiting subdifferential} of $h$ at $x$ is defined by 
\begin{equation*}
\partial_L h(x) :=\left\{u\in \mathcal{H}:\; \exists x_n\st{h}{\to}x,\; u_n\to u \text{~~with~~} u_n\in \widehat{\partial} h(x_n)\right\}.
\end{equation*}
We set $\widehat{\partial} h(x) =\partial_L h(x) :=\varnothing$ when $|h(x)| =+\infty$ and define $\dom\partial_L h :=\{x\in \mathcal{H}: \partial_L h(x)\neq \varnothing\}$. It follows from the definition that the limiting subdifferential has the \emph{robustness property}
\begin{equation}\label{e:robustness}
\partial_L h(x) =\left\{u\in \mathcal{H}:\; \exists x_n \st{h}{\to}x,\; u_n\to u \text{~~with~~} u_n\in \partial_L h(x_n)\right\}.
\end{equation}
For a convex function $h$, both Fr\'echet and limiting subdifferentials reduce to the classical subdifferential in convex analysis (see, for example, \cite[Theorem 1.93]{Mor06})
\begin{equation*}
\partial h(x): =\left\{u\in \mathcal{H}:\; \forall z\in \mathcal{H},\; \langle u,z-x\rangle\leq h(z)-h(x)\right\}.
\end{equation*}
Moreover, for a strictly differentiable\footnote{A function $h$ is strictly differentiable at $x$ if there exists $u \in \mathcal{H}$ such that $\lim\limits_{y,z \to x}\frac{h(y)-h(z)-\langle u,y-z\rangle}{\|y-z\|} =0$. Clearly, if $h$ is continuously differentiable at $x$, then it is strictly differentiable at $x$.} function $h$, both Fr\'echet and limiting subdifferentials reduce to the derivative of $h$ denoted by $\nabla h$.

Let $S$ be a nonempty subset of $\mathcal{H}$. Its convex hull is denoted by ${\rm conv} \, S$. The indicator function of $S$ is given by 
\begin{equation}\label{eq:indicator}
\iota_S(x) :=\begin{cases}
0 & \text{if~} x \in S, \\
+\infty & \text{if~} x \notin S.
\end{cases}
\end{equation}
Given $x\in \mathcal{H}$, the \emph{Fr\'echet normal cone} of $S$ at $x$ is given by $\widehat{N}_S(x) :=\widehat{\partial} \iota_S(x)$ and the \emph{limiting normal cone} of $S$ at $x$ is $N_S(x) :=\partial_L \iota_S(x)$. The set $S$ is \emph{regular} at $x\in S$ if $N_S(x) =\widehat{N}_S(x)$. We say that $S$ is regular if it is regular at all of its points. It is known, e.g., from \cite[Proposition~1.5]{Mor06} that $S$ is regular at $x\in S$ if it is locally convex around $x$, i.e., if there exists a neighborhood $U$ of $x$ such that $S\cap U$ is convex.

For a function $h\colon \mathcal{H}\to \left[-\infty,+\infty\right]$ finite at $x$, we say that $h$ is \emph{regular}\footnote{In some literature, this is also referred as \emph{lower regular} in \cite{Mor06,MNY06}.} at $x$ if $\widehat{\partial} h(x) =\partial_L h(x)$. For a  proper lower semicontinuous function $h$, it is clear that if $h$ is convex around $x$ or strictly differentiable at $x$, then it is regular at $x$. In the case where $h$ is an indicator function of a closed set or is a Lipschitz continuous function around $x$, according to \cite[Proposition~1.92]{Mor06}, $h$ is regular at $x$ if and only if $\epi h :=\{(x,r)\in \mathcal{H}\times \mathbb{R}: r\geq h(x)\}$ is regular at $(x,h(x))$.

In general, the limiting subdifferential set can be nonconvex (e.g., for $h(x)=-|x|$ at $0\in\mathbb{R}$) while $\partial_L h$ enjoys comprehensive calculus rules based on \emph{variational/extremal principles} of variational analysis \cite{Mor06,RW98}. In particular, the following sum rule and quotient rule and for limiting subdifferential will be useful for us later.

\begin{lemma}[Sum and quotient rules]
\label{l:calculus}
Let $f,g\colon \mathcal{H}\to \left(-\infty,+\infty\right]$ be proper lower semicontinuous functions, and let $x \in \mathcal{H}$.
Then the following hold:
\begin{enumerate}
\item
\label{l:calculus_sum}
Suppose that $f$ is finite at $\overline{x}$ and $g$ is locally Lipschitz around $\overline{x}$. Then $\partial_L(f+g)(\overline{x}) \subseteq \partial_L f(\overline{x}) + \partial_L g(\overline{x})$, where the equality holds if both $f$ and $g$ are regular at $\overline{x}$, in which case $f+g$ is also regular at $\overline{x}$. Moreover, if $g$ is strictly differentiable at $\overline{x}$, then $\partial_L(f+g)(\overline{x}) =\partial_L f(\overline{x}) +\nabla g(\overline{x})$.
\item
\label{l:calculus_quotient}
Suppose that $f$ and $g$ are Lipschitz continuous around $\overline{x}$, and $g(\overline{x})\neq 0$. Then, if $\widehat{\partial} g$ is nonempty-valued around $\overline{x}$, one has
\begin{equation}\label{e:quotient}
\partial_L\left(\frac{f}{g}\right)(\overline{x}) \subseteq \frac{\partial_L(g(\overline{x})f)(\overline{x}) -f(\overline{x})\partial_L g(\overline{x})}{g(\overline{x})^2}.
\end{equation}
If $g$ is strictly differentiable at $\overline{x}$, one has
\begin{equation}\label{e:quotient'}
\partial_L\left(\frac{f}{g}\right)(\overline{x}) =\frac{\partial_L(g(\overline{x})f)(\overline{x}) -f(\overline{x})\nabla g(\overline{x})}{g(\overline{x})^2},
\end{equation}
and consequently $f/g$ is regular at $\overline{x}$ if and only if the function $x\mapsto g(\overline{x})f(x)$ is regular at $\overline{x}$.
\end{enumerate}
\end{lemma}
\begin{proof}
\ref{l:calculus_sum}: We first derive from \cite[Theorem~3.36]{Mor06} and its following remark that $\partial_L(f+g)(\overline{x}) \subseteq \partial_L f(\overline{x}) + \partial_L g(\overline{x})$ and that if both $f$ and $g$ are regular at $\overline{x}$, then so is $f+g$ and $\partial_L(f+g)(\overline{x}) =\partial_L f(\overline{x}) + \partial_L g(\overline{x})$. By \cite[Proposition~1.107(ii)]{Mor06}, this equality also holds if $g$ is strictly differentiable at $\overline{x}$.  

\ref{l:calculus_quotient}: As $f$ and $g$ are Lipschitz continuous around $\overline{x}$ and $g(\overline{x})\neq 0$,  \cite[Proposition~1.111(ii)]{Mor06} implies that
\begin{equation}\label{e:quotient0}
\partial_L\left(\frac{f}{g}\right)(\overline{x}) =\frac{\partial_L(g(\overline{x})f -f(\overline{x})g)(\overline{x})}{g(\overline{x})^2}.
\end{equation}
Thus, to see \eqref{e:quotient}, it suffices to show that $\partial_L(g(\overline{x})f -f(\overline{x})g)(\overline{x})\subseteq \partial_L(g(\overline{x})f)(\overline{x}) -f(\overline{x})\partial_L g(\overline{x})$. This is obvious if $f(\overline{x}) =0$. If $f(\overline{x}) <0$, then $-f(\overline{x}) >0$ and, by \ref{l:calculus_sum},
\begin{equation*}
\partial_L(g(\overline{x})f -f(\overline{x})g)(\overline{x})\subseteq \partial_L(g(\overline{x})f)(\overline{x}) +\partial_L(-f(\overline{x})g)(\overline{x}) =\partial_L(g(\overline{x})f)(\overline{x}) -f(\overline{x})\partial_L g(\overline{x}).
\end{equation*}
If $f(\overline{x}) >0$, then $\widehat{\partial}(f(\overline{x})g) =f(\overline{x})\widehat{\partial} g$ is nonempty-valued around $\overline{x}$ and, by \cite[Corollary~3.4]{MNY06},
\begin{equation*}
\partial_L(g(\overline{x})f -f(\overline{x})g)(\overline{x})\subseteq \partial_L(g(\overline{x})f)(\overline{x}) -\partial_L(f(\overline{x})g)(\overline{x}) =\partial_L(g(\overline{x})f)(\overline{x}) -f(\overline{x})\partial_L g(\overline{x}),
\end{equation*}
from which we get the claimed inclusion. The equality \eqref{e:quotient'} then follows from \eqref{e:quotient0} and \ref{l:calculus_sum}. Finally, the conclusion for the regularity of $f/g$ follows from \eqref{e:quotient'} and \cite[Corollaries~1.12.2 and 1.14.2]{Kru03}.
\end{proof}

We say that a function $h$ is \emph{weakly convex (on $\mathcal{H}$)} if there exists $\rho \geq 0$ such that $h+\frac{\rho}{2}\|\cdot\|^2$ is a convex function. Moreover, the smallest constant $\rho$ such that $h+\frac{\rho}{2}\|\cdot\|^2$ is convex is called the \emph{modulus} for a weakly convex function $h$. More generally, a function $h$ is said to be \emph{weakly convex on $S \subseteq \mathcal{H}$ with modulus $\rho$} if $h+\iota_S$ is weakly convex with modulus $\rho$. Weakly convex functions form a broad class of functions which covers quadratic functions, convex functions and continuously differentiable functions whose gradient is Lipschitz continuous. Recall that the (one-sided) \emph{directional derivative} of $h$ in the direction $d$ is defined by
\begin{equation*}
h'(x;d) =\lim_{t \to 0^+} \frac{h(x+td)-h(x)}{t},
\end{equation*}
provided the limit exists. We end this section with the following lemma.

\begin{lemma}
\label{l:weaklycvx}
Let $S$ be a nonempty closed convex subset of $\mathcal{H}$, let $\overline{x}\in S$, and let $h\colon \mathcal{H}\to \left(-\infty,+\infty\right]$ be a proper lower semicontinuous function which is weakly convex on $S$. Then the following hold:
\begin{enumerate}
\item\label{l:weaklycvx_cvx}
For all $x\in \mathcal{H}$, $\partial_L(h+\iota_S)(x)$ is a (possibly empty) closed convex set.
\item\label{l:weaklycvx_dderivative}
If $\overline{x}\in \inte S$ and $h$ is continuous at $\overline{x}$, then $\partial_L h(\overline{x}) =\partial_L(h+\iota_S)(\overline{x})\neq \varnothing$ and, for all $x\in S$, $h'(\overline{x};x-\overline{x}) =\max\{\scal{v}{x-\overline{x}}: v\in \partial_L(h+\iota_S)(\overline{x})\}$. In particular, if $h$ is a weakly convex function on $\mathcal{H}$ which is continuous at $\overline{x}$, then, for all $d\in \mathcal{H}$, $h'(\overline{x};d) =\max\{\scal{v}{d}: v\in \partial_L h(\overline{x})\}$. 
\item\label{l:weaklycv_equi}
$0\in \partial_L(h+\iota_S)(\overline{x})$ if and only if, for all $x\in S$, $h'(\overline{x};x-\overline{x})\geq 0$.
\end{enumerate}
\end{lemma}
\begin{proof}
By assumption, there exists $\rho \geq 0$ such that $H :=h +\iota_S +\frac{\rho}{2}\|\cdot\|^2$ is a convex function. Using Lemma~\ref{l:calculus}\ref{l:calculus_sum}, we have that, for all $x\in \mathcal{H}$, $\partial H(x) =\partial_L H(x) =\partial_L(h+\iota_S)(x) +\rho x$, and so $\partial_L(h+\iota_S)(x) =\partial H(x) -\rho x$. Since $S$ is convex, it follows from the definition of directional derivative that, for all $x\in \mathcal{H}$,
\begin{equation}\label{e:dderivative}
h'(\overline{x};x-\overline{x}) \leq (h+\iota_S)'(\overline{x};x-\overline{x}) =H'(\overline{x};x-\overline{x}) -\scal{\rho \overline{x}}{x-\overline{x}},
\end{equation}
where the first inequality is an equality if $x\in S$.

\ref{l:weaklycvx_cvx}: Since $\partial H(x)$ is a closed convex set, so is $\partial_L(h+\iota_S)(x)$.

\ref{l:weaklycvx_dderivative}: Assume that $\overline{x}\in \inte S$ and $h$ is continuous at $\overline{x}$, then $h(x) =(h+\iota_S)(x)$ for $x$ near $\overline{x}$ and $H$ is a convex function which is continuous at $\overline{x}\in \inte S$ and $h$ is continuous at $\overline{x}$, and hence $\partial_L h(\overline{x}) =\partial_L(h+\iota_S)(\overline{x}) =\partial_L H(\overline{x}) -\rho \overline{x}\neq \varnothing$.

Now, let $x\in S$. By \cite[Theorem~17.18]{BC17}, $H'(\overline{x};x-\overline{x}) =\max\{\scal{u}{x-\overline{x}}: u\in \partial H(\overline{x})\} =\max\{\scal{v+\rho \overline{x}}{x-\overline{x}}: v\in \partial (h+\iota_S)(\overline{x})\}$, which combined with \eqref{e:dderivative} implies the desired claim.

\ref{l:weaklycv_equi}: Set $H_1 :=h +\iota_S +\frac{\rho}{2}\|\cdot-\overline{x}\|^2$. Then $H_1$ is also a convex function. We derive from Lemma~\ref{l:calculus}\ref{l:calculus_sum} that $\partial_L(h+\iota_S)(\overline{x}) =\partial H_1(\overline{x})$ and from \eqref{e:dderivative} that $h'(\overline{x};x-\overline{x})\geq 0$ for all $x\in S$ if and only if  $(h+\iota_S)'(\overline{x};x-\overline{x}) =H_1'(\overline{x};x-\overline{x})\geq 0$ for all $x\in \mathcal{H}$. The conclusion then follows from \cite[Theorem~16.3 and Proposition~17.3]{BC17}.
\end{proof}

\subsection*{Kurdyka--{\L}ojasiewicz property} 

Next, we recall the celebrated Kurdyka--{\L}ojasiewicz (KL) property \cite{Kur98,Loj63} which plays an important role in our convergence analysis later on. For each $\eta \in \left(0,+\infty\right]$, we denote by $\Phi_\eta$ the class of all continuous concave functions $\varphi\colon \left[0, \eta\right)\to \RP$ such that $\varphi(0) =0$ and $\varphi$ is continuously differentiable on $\left(0, \eta\right)$ with $\varphi' >0$.

Let $h\colon \mathcal{H}\to \left(-\infty, +\infty\right]$ be a proper lower
semicontinuous function. We say that $h$ satisfies the \emph{KL property} \cite{Kur98,Loj63} at $\overline{x}\in \dom\partial_L h$ if there exist a neighborhood $U$ of $\overline{x}$, $\eta \in \left(0,+\infty\right]$, and a function $\varphi\in \Phi_\eta$ such that, for all $x\in U$ with $h(\overline{x}) <h(x) <h(\overline{x})+\eta$, one has
\begin{equation*}
\varphi'(h(x)-h(\overline{x}))\dist(0,\partial_L h(x))\geq 1.
\end{equation*}
If $h$ satisfies the KL property at each point in $\dom\partial_L h$, then $h$ is called a \emph{KL function}. For a function $h$ satisfying the KL property at $\overline{x}\in \dom\partial_L h$, if the corresponding function $\varphi$ can be chosen as $\varphi(s) =\gamma s^{1-\alpha}$ for some $\gamma\in \RPP$ and $\alpha\in \left[0,1\right)$, then we say that $h$ has the \emph{KL property at $\overline{x}$ with an exponent of $\alpha$}. If $h$ is a KL function and has the same exponent $\alpha$ at any $\overline{x}\in \dom\partial_L h$, then $h$ is called a \emph{KL function with an exponent of $\alpha$}.

This definition encompasses broad classes of functions that arise in practical optimization problems. For example, it is known that if $h$ is a proper lower semicontinuous semi-algebraic function, then $h$ is a KL function with a suitable exponent of $\alpha\in \left[0,1\right)$. The semi-algebraic function covers many common nonsmooth functions that appear in modern optimization problems such as functions which can be written as maximum or minimum of finitely many polynomials, Euclidean norms and the eigenvalues and rank of a matrix. Also, sums, products, and quotients of semi-algebraic functions are still semi-algebraic. For some recent development of KL property, see  \cite{AB09,LP18}.

\begin{lemma}
\label{l:uniformKL}
Let $(x_n)_{n\in \mathbb{N}}$ be a bounded sequence in $\mathcal{H}$, let $\Omega$ be the set of cluster points of $(x_n)_{n\in \mathbb{N}}$, and let $h\colon \mathcal{H}\to \left(-\infty, +\infty\right]$ be a proper lower semicontinuous function that is constant on $\Omega$ and satisfies the KL property at each point of $\Omega$. Set $\Omega_0 :=\{\overline{x}\in \Omega: h(x_n)\to h(\overline{x}) \text{~as~} n\to +\infty\}$ and suppose that $\Omega_0\neq \varnothing$. Then there exist $\eta \in \left(0,+\infty\right]$, $\varphi\in \Phi_\eta$, and $n_0\in \mathbb{N}$ such that, for all $\overline{x}\in \Omega_0$,
\begin{equation}
\varphi'(h(x_n)-h(\overline{x}))\dist(0,\partial_L h(x_n))\geq 1
\end{equation}
whenever $n\geq n_0$ and $h(x_n) >h(\overline{x})$. Moreover, if $h$ satisfies the KL property at every point of $\Omega$ with an exponent of $\alpha$, then the function $\varphi$ can be chosen as $\varphi(s) =\gamma s^{1-\alpha}$ for some $\gamma\in \RPP$.
\end{lemma}
\begin{proof}
Since $(x_n)_{n\in \mathbb{N}}$ is bounded, $\Omega$ is nonempty and compact. According to \cite[Lemma~6]{BST14}, there exists $\varepsilon >0$, $\eta >0$, and $\varphi\in \Phi_\eta$ such that
\begin{equation}\label{e:KLine}
\varphi(h(x)-h(\overline{x}))\dist(0,\partial_L h(x))\geq 1
\end{equation}
whenever $\dist(x,\Omega) <\varepsilon$ and $h(\overline{x}) <h(x) <h(\overline{x}) +\eta$. From the proof of \cite[Lemma~6]{BST14}, we also see that, if $h$ satisfies the KL property at every point of $\Omega$ with an exponent of $\alpha$, then the function $\varphi$ can be chosen as $\varphi(s) =\gamma s^{1-\alpha}$ for some $\gamma\in \RPP$.

We note that $\dist(x_n,\Omega)\to 0$ as $n\to +\infty$. Indeed, suppose otherwise. Then there exist $\overline{\varepsilon} >0$ and a subsequence $(x_{k_n})_{n\in \mathbb{N}}$ of $(x_n)_{n\in \mathbb{N}}$ such that, for all $n\in \mathbb{N}$, $\dist(x_{k_n},\Omega)\geq \overline{\varepsilon}$. Since $(x_{k_n})_{n\in \mathbb{N}}$ is also bounded, there exists a subsequence $(x_{l_{k_n}})_{n\in \mathbb{N}}$ such that $x_{l_{k_n}}\to x^*$. We have that $x^*\in \Omega$ and that, for all $n\in \mathbb{N}$, $\dist(x_{l_{k_n}},\Omega)\geq \overline{\varepsilon}$. By the continuity of the distance function (see, e.g., \cite[Example~1.48]{BC17}), $\dist(x^*,\Omega)\geq \overline{\varepsilon}$, which contradicts the fact that $x^*\in \Omega$.

Now, let $\overline{x}\in \Omega_0$. Since $\dist(x_n,\Omega)\to 0$ and $h(x_n)\to h(\overline{x})$ as $n\to +\infty$, one can find $n_0\in \mathbb{N}$ such that, for all $n\geq n_0$,
\begin{equation*}
\dist(x_n,\Omega) <\varepsilon \quad\text{and}\quad
h(x_n) <h(\overline{x}) +\eta.
\end{equation*}
Here, we note that $n_0$ does not depend on $\overline{x}$ because $h(\overline{x})$ is independent of $\overline{x}\in \Omega_0\subseteq \Omega$. The conclusion follows from \eqref{e:KLine} and its following remark.
\end{proof}

\section{Stationary points of fractional programs}

In this section, we introduce various versions of stationary points for fractional programs and examine their relationships.
\begin{definition}[Stationary points, lifted stationary points \& strong lifted stationary points]
For problem~\eqref{e:prob}, we say that $\overline{x}\in S$ is
\begin{enumerate}
\item
a \emph{(limiting) stationary point} if
$0 \in \partial_L(\frac{f}{g}+\iota_S)(\overline{x})$;
\item
a \emph{(limiting) lifted stationary point} if
$0 \in g(\overline{x})\partial_L(f+\iota_S)(\overline{x})-f(\overline{x})\partial_Lg(\overline{x});$
\item
a \emph{(limiting) strong lifted stationary point} if
$f(\overline{x})\partial_Lg(\overline{x}) \subseteq g(\overline{x})\partial_L(f+\iota_S)(\overline{x})$.
\end{enumerate}
\end{definition}

It is well known that a necessary condition for $\overline{x}\in S$ to be a local minimizer of $\frac{f}{g}$ on $S$ is $0\in \partial_L(\frac{f}{g}+\iota_S)(\overline{x})$. Thus, any local minimizer must be a stationary point. Next, we examine the relationships between the above three versions of stationary points.

\begin{lemma}[Stationary points vs. lifted stationary points]
\label{l:stationary}
Consider problem~\eqref{e:prob} in which $f,g\colon \mathcal{H}\to \left(-\infty,+\infty\right]$ are proper lower semicontinuous functions and $S$ is a nonempty closed subset of $\mathcal{H}$. Let $C$ be a nonempty closed subset of $\mathcal{H}$ such that $C\cap S\neq \varnothing$ and let $\overline{x}\in C\cap S$. Suppose that $g(\overline{x}) >0$ and that $f =f_1 +\iota_C$, where one of the following is satisfied:
\begin{enumerate}[label =(\alph*)]
\item\label{fLip}
$f_1$ is Lipschitz continuous around $\overline{x}$ and $\overline{x}\in \inte(C\cap S)$;
\item\label{fLip&reg}
$f_1$ is Lipschitz continuous around $\overline{x}$, $f_1$ and $C\cap S$ are regular at $\overline{x}$, and $g$ is positive around $\overline{x}$;
\item\label{fDiff}
$f_1$ is strictly differentiable at $\overline{x}$ and $g$ is positive around $\overline{x}$.
\end{enumerate}
Then the following statements hold:
\begin{enumerate}
\item\label{l:stationary_imply}
If $g$ is Lipschitz continuous around $\overline{x}$ and $\widehat{\partial} g$ is nonempty-valued around $\overline{x}$, then
\begin{equation}
\label{e:f/g+i_S}
\partial_L\left(\frac{f}{g}+\iota_S\right)(\overline{x}) \subseteq \frac{g(\overline{x})\partial_L(f+\iota_S)(\overline{x})-f(\overline{x})\partial_Lg(\overline{x})}{g(\overline{x})^2},
\end{equation}
in which case, if $\overline{x}$ is a stationary point of \eqref{e:prob}, then it is a lifted stationary point of \eqref{e:prob}.
\item\label{l:stationary_equi}
If $g$ is strictly differentiable at $\overline{x}$, then
\begin{equation}
\label{e:f/g+i_S=}
\partial_L\left(\frac{f}{g}+\iota_S\right)(\overline{x})
=\frac{g(\overline{x})\partial_L(f+\iota_S)(\overline{x})-f(\overline{x})\nabla g(\overline{x})}{g(\overline{x})^2},
\end{equation}
in which case, $\overline{x}$ is a stationary point of \eqref{e:prob} if and only if it is a lifted stationary point of \eqref{e:prob}.
\end{enumerate}
\end{lemma}
\begin{proof}
\ref{l:stationary_imply}: If \ref{fLip} holds, then $\partial_L(\frac{f}{g}+\iota_S)(\overline{x}) =\partial_L(\frac{f+\iota_S}{g})(\overline{x})$ and $f+\iota_S =f_1 +\iota_{C\cap S}$ is Lipschitz continuous around $\overline{x}$, hence \eqref{e:f/g+i_S} holds due to Lemma~\ref{l:calculus}\ref{l:calculus_quotient} and the fact that $g(\overline{x}) >0$.

In both cases \ref{fLip&reg} and \ref{fDiff}, $f_1$ is Lipschitz continuous at $\overline{x}$, and so is $f_1/g$. Using the fact that $g$ is positive around $\overline{x}$ and applying Lemma~\ref{l:calculus}\ref{l:calculus_sum} and then Lemma~\ref{l:calculus}\ref{l:calculus_quotient}, we have
\begin{equation*}
\partial_L\left(\frac{f}{g}+\iota_S\right)(\overline{x}) =\partial_L\left(\frac{f_1}{g}+\iota_{C\cap S}\right)(\overline{x}) \subseteq \partial_L\left(\frac{f_1}{g}\right)(\overline{x}) +\partial_L \iota_{C\cap S}(\overline{x})
\subseteq \frac{g(\overline{x})\partial_L f_1(\overline{x}) -f_1(\overline{x})\partial_L g(\overline{x})}{g(\overline{x})^2} +\partial_L \iota_{C\cap S}(\overline{x}).
\end{equation*}
As $f_1$ and $C\cap S$ are regular at $\overline{x}$ (if \ref{fLip&reg} holds) or $f_1$ is strictly differentiable at $\overline{x}$ (if \ref{fDiff} holds), also by Lemma~\ref{l:calculus}\ref{l:calculus_sum}, $\partial_L f_1(\overline{x}) +\partial_L \iota_{C\cap S}(\overline{x}) =\partial_L(f_1+\iota_{C\cap S})(\overline{x}) =\partial_L(f+\iota_S)(\overline{x})$. Noting that $f_1(\overline{x}) =f(\overline{x})$ and that $\partial_L \iota_{C\cap S}(\overline{x}) =\frac{1}{g(\overline{x})}\partial_L \iota_{C\cap S}(\overline{x})$ since $g(\overline{x}) >0$, we also obtain \eqref{e:f/g+i_S}.

\ref{l:stationary_equi}: As $g$ is strictly differentiable at $\overline{x}$ with $g(\overline{x}) \neq 0$, we note that $f_1/g$ is regular at $\overline{x}$ if $f_1$ is regular at $\overline{x}$ (by Lemma~\ref{l:calculus}\ref{l:calculus_quotient}), and that $f_1/g$ is strictly differentiable at $\overline{x}$ if $f_1$ is strictly differentiable at $\overline{x}$. Now, \eqref{e:f/g+i_S=} is obtained by using the same argument as in \ref{l:stationary_imply} and noting that the inclusions become equalities due to the strict differentiability of $g$ (in all of three cases), the regularity of $f_1/g$ and $C\cap S$ (in the case of \ref{fLip&reg}), or the strict differentiability of $f_1/g$ (in the case of \ref{fDiff}).
\end{proof}

From the definition, any strong lifted stationary point $\overline{x}$ with $\partial_L g(\overline{x})\neq \varnothing$ is also a lifted stationary point. Moreover, if $g$ is strictly differentiable, then strong lifted stationary points and lifted stationary points are the same. However, if $g$ is not strictly differentiable, then a lifted stationary point need not to be a strong lifted stationary point in general, as in the following example.

\begin{example}
Consider the following one-dimensional fractional program
\begin{equation}\label{e:ex}
\min_{x \in [-1,1]} \frac{x^2+1}{|x|+1}.
\end{equation}
Let $\overline{x}=0$, $f(x)=x^2+1$,  $g(x)=|x|+1$ and $S=[-1,1]$. Clearly, $\partial_L (f+\iota_S)(\overline{x})=\{0\}$ and $\partial_L g(\overline{x})=[-1,1]$. Then, $\overline{x}$ is a lifted stationary point because $0 \in g(\overline{x})\partial_L(f+\iota_S)(\overline{x})-f(\overline{x})\partial_Lg(\overline{x})=[-1,1]
$. On the other hand, $\overline{x}$ is not a strong lifted stationary point as
\begin{equation*}
[-1,1] = f(\overline{x})\partial_Lg(\overline{x}) \nsubseteq g(\overline{x})\partial_L(f+\iota_S)(\overline{x})=\{0\}.
\end{equation*}
Indeed, a direct verification shows that the lifted stationary points of \eqref{e:ex} are $-\sqrt{2}+1$, $0$, and $\sqrt{2}-1$; while the set of strong lifted stationary points of \eqref{e:ex} is $\{-\sqrt{2}+1,\sqrt{2}-1\}$, which coincides with the set of local/global minimizers of problem \eqref{e:ex}.
\end{example}

Finally, we establish the relationship between the strong lifted stationary points and the recently studied d(irectional)-stationary points in the difference-of-convex (DC) optimization literature \cite{PRA16,BH20}. Recall that $\overline{x}\in S$ is a \emph{d-stationary point} of a function $h$ on $S$ if, for all $x\in S$, $h'(\overline{x};x-\overline{x})\geq 0$.

\begin{lemma}[Strong lifted stationary point vs. d-stationary points]
Consider problem~\eqref{e:prob} in which $S$ is a nonempty closed convex subset of $\mathcal{H}$ and both $f+\iota_S$ and $g$ are proper lower semicontinuous weakly convex functions. Let $\overline{x}\in S$. Suppose that $g$ is continuous on an open set containing $S$ and that $g(\overline{x}) >0$. Then $\overline{x}$ is a strong lifted stationary point of \eqref{e:prob} if and only if it is a d-stationary point of $f -\frac{f(\overline{x})}{g(\overline{x})}g$ on $S$.
\end{lemma}
\begin{proof}
Set $\overline{\theta} :=f(\overline{x})/g(\overline{x})$. First, $\overline{x}$ is a strong lifted stationary point of \eqref{e:prob} if and only if, for all $v\in \partial_L g(\overline{x})$, $0\in \partial_L(f+\iota_S)(\overline{x}) -\overline{\theta}v =\partial_L(f-\overline{\theta}\scal{v}{\cdot}+\iota_S)(\overline{x})$, which is equivalent to, for all $v\in \partial_L g(\overline{x})$ and all $x\in S$, $f'(\overline{x};x-\overline{x}) -\overline{\theta}\scal{v}{x-\overline{x}} =(f-\scal{\overline{\theta}v}{\cdot})'(\overline{x};x-\overline{x})\geq 0$ due to Lemma~\ref{l:weaklycvx}\ref{l:weaklycv_equi}. Now, as $g$ is weakly convex on $\mathcal{H}$ and continuous at $\overline{x}$, applying the last conclusion of Lemma~\ref{l:weaklycvx}\ref{l:weaklycvx_dderivative} with $h=g$, we have, for all $x\in \mathcal{H}$, $g'(\overline{x};x-\overline{x}) =\max\{\scal{v}{x-\overline{x}}: v\in \partial_L g(\overline{x})\}$, which completes the proof.
\end{proof}

\section{Extrapolated proximal subgradient (e-PSG) algorithm}
\label{s:ePSG}

In this section, we consider problem~\eqref{e:prob} under the following assumptions.

\begin{assumption}
\label{a:f}
$f =f^\mathfrak{s}+f^\mathfrak{n}$, where $f^\mathfrak{s}$ is a continuously differentiable convex function whose gradient $\nabla f^\mathfrak{s}$ is Lipschitz continuous with modulus $\ell$ on $\mathcal{H}$, and $f^\mathfrak{n}$ is a proper lower semicontinuous function, $S\cap \dom f\neq \varnothing$ and, for all $x\in S\cap \dom f$, $f(x)\geq 0$.
\end{assumption}

\begin{assumption}
\label{a:g}
$g$ is a proper lower semicontinuous function which is finite and positive on $S$, continuous on an open set containing $S$, and either weakly convex with modulus $\beta$ on an open convex set containing $S$, or regular and weakly convex with modulus $\beta$ on $S$.
\end{assumption}

We note that the nonnegative assumption of the numerator $f$ and the positivity assumption of the denominator $g$ are standard in the literature of fractional programs \cite{BotCse17,CFS85,Din67}. Also, these assumptions are easily satisfied for many practical optimization models in diverse areas, in particular, for all the motivating examples we mentioned in the introduction.
We now propose the following proximal subgradient algorithm with extrapolation for solving the nonsmooth and nonconvex fractional programming problem~\eqref{e:prob}. To do this, we define the following boundedness condition (BC):  There exist $m, M\in \RPP$ such that, for all $x\in S\cap \dom f$,
\begin{equation}\label{a:bound}
m \leq g(x) \leq M. \tag{BC}
\end{equation}

\begin{tcolorbox}[
	left=0pt,right=0pt,top=0pt,bottom=0pt,
	colback=blue!10!white, colframe=blue!50!white,
  	boxrule=0.2pt,
  	breakable]
\begin{algorithm}[Extrapolated proximal subgradient (e-PSG) algorithm]
\label{algo:main}
\step{}\label{step:initial}
Choose $x_{-1} =x_0\in S\cap \dom f$ and set $n =0$. Let $\delta\in \RPP$, let $\zeta\in \RPP$ be such that $1-\sqrt{\beta}\zeta >0$, and let
\begin{equation*}
\overline{\mu}\in \left[0, \frac{\delta(1-\sqrt{\beta}\zeta)\sqrt{mM}}{2M}\right)
\quad\text{and}\quad
\overline{\kappa}\in \left[0,\sqrt{\frac{m\delta(1-\sqrt{\beta}\zeta)}{\ell M}-\frac{2m\overline{\mu}}{\ell\sqrt{mM}}} \, \right),
\end{equation*}
where $\ell$ is defined in Assumption~\ref{a:f}, $\beta$ is defined in Assumption~\ref{a:g}, while $m$ and $M$ are given in \eqref{a:bound}. In the absence of \eqref{a:bound}, we let $\overline{\mu} =0$ and $\overline{\kappa} =0$.

\step{}\label{step:main}
Set $\theta_n =\frac{f(x_n)}{g(x_n)}$, let $g_n\in \partial_Lg(x_n)$, choose $\tau_n\in \mathbb{R}$ such that $0< \tau_n\leq 1/\max\{\sqrt{\beta}\theta_n/\zeta, \delta\}$.
Let $u_n =x_n +\kappa_n(x_n-x_{n-1})$ with $\kappa_n\in [0,\overline{\kappa}]$, $v_n =x_n +\mu_n(x_n-x_{n-1})$ with $\mu_n\in [0,\overline{\mu}\tau_n]$, and find
\begin{equation*}
x_{n+1} \in \argmin_{x\in S} \left(f^\mathfrak{n}(x) +f^\mathfrak{s}(u_n) +\scal{\nabla f^\mathfrak{s}(u_n)}{x-u_n} +\frac{1}{2\tau_n}\|x-v_n-\tau_n\theta_ng_n\|^2 +\frac{\ell}{2}\|x-u_n\|^2 \right).
\end{equation*}

\step{}
If a termination criterion is not met, let $n =n+1$ and go to Step~\ref{step:main}.
\end{algorithm}
\end{tcolorbox}

Before proceeding, we first make a few observations. Firstly, in the special case where $f^\mathfrak{s}\equiv 0$, $f^\mathfrak{n}$ is convex, $\kappa_n = 0$, $\mu_n=0$, $\ell=0$ and $g$ is continuously differentiable (and so, $g_n=\nabla g(x_n)$), Algorithm~\ref{algo:main} reduces to the proximal gradient algorithm proposed in \cite{BotCse17}. Secondly, in Step~\ref{step:main}, the part ``$f^\mathfrak{s}(u_n) +\scal{\nabla f^\mathfrak{s}(u_n)}{x-u_n}$'' serves as the linear approximation of $f^{\mathfrak{s}}$ at $u_n$. Although the
term ``$f^\mathfrak{s}(u_n)$'' can be removed as it does not contribute to the minimization problem, we prefer to leave it here for understanding the algorithm intuitively. Finally, it is worth noting
that when $\overline{\mu} < \frac{\delta(1-\sqrt{\beta}\zeta)\sqrt{mM}}{2M}$, then $\frac{m\delta(1-\sqrt{\beta}\zeta)}{\ell M}>\frac{2m\overline{\mu}}{\ell\sqrt{mM}}$, and so, the
choice of $\overline{\kappa}$ in Step~\ref{step:initial} makes sense.

\begin{remark}[Discussions on computing the subproblems]
In the above algorithm, the major computational cost lies in solving the subproblem in Step~\ref{step:main}. In Step~\ref{step:main}, finding $x_{n+1}$ is indeed equivalent to computing the proximal operator\footnote{The proximal operator of a function $h$ is denoted by ${\rm Prox}_h$ and is defined as ${\rm Prox}_h(x)= {\rm argmin}_y\{h(y)+\frac{1}{2}\|y-x\|^2\}$.} of $\frac{\tau_n}{1+ \ell \tau_n }\left(f^\mathfrak{n}+\iota_S\right)$ at the point $\frac{v_n+\tau_n\theta_ng_n+ \ell\tau_nu_n -\tau_n \nabla f^{\mathfrak{s}}(u_n)}{1+\ell\tau_n}$,
 where $f^\mathfrak{n}$ is the nonsmooth part of the numerator. This can be done efficiently for functions $f$ and sets $S$ with specific structures. For example,
\vspace{-0.2cm}
\begin{enumerate}
\item
In the case where $S$ is a polyhedral and $f^\mathfrak{n}$ is the maximum of finitely may affine functions, the optimization problem in Step~\ref{step:main} can be reformulated as a convex quadratic optimization problem with linear constraints, and so, can be solved by calling a QP solver. This, in particular, covers the motivating examples  \ref{r:ex_ssr} and \ref{r:ex_Sharpe} in the introduction.

\item
In the case of $S =\mathbb{R}^N$, $f^\mathfrak{s}$ is a convex quadratic function, $f^\mathfrak{n} =\iota_C$ with $C$ being the unit sphere (as in the motivating example~\ref{r:rayleigh}  in the introduction), the optimization problem in Step~\ref{step:main} reduces to computing the projection onto the unit sphere $C$ which  has a closed form solution.

\item
In the case of $f^\mathfrak{n}$ is the minimum of finitely many (nonconvex) quadratic function, that is, $f^\mathfrak{n}(x)=\min_{1 \leq i \leq m}\{x^\top A_ix+a_i^\top x+\alpha_i\}$ and $S=\{x: \|x\|^2 \leq \rho\}$, the optimization problem in Step~\ref{step:main} can be computed by solving $m$ many (nonconvex) quadratic optimization problem with a ball constraint. As each quadratic optimization problem with a ball constraint is a trust-region problem, and can be equivalently reformulated as either a
semi-definite program (SDP) or an eigenvalue problem. So, the subproblem can be solved by calling an SDP solver or an eigenvalue problem  solver.

\item
In the case of $S=\{x:q_i(x) \leq 0, i=1,\dots,m_1\}$ where $q_i$ are convex quadratic functions, and $f^\mathfrak{n}(x)=\max_{1 \leq i\leq m_2} h_i(x)$ where each $h_i$ is a convex quadratic function, the optimization problem in Step~\ref{step:main} can be reformulated as a convex quadratic optimization problem with convex quadratic constraints, and so, can be further rewritten as a semidefinite programming problem (SDP) and solved by calling an SDP solver.
\end{enumerate}
\end{remark}

\begin{remark}[Discussions of the extrapolation parameter]
We first note that our choice of the extrapolation parameter covers the popular extrapolation parameter used for restarted FISTA in the case where $g$ is convex (see, for example, \cite[Chapter~10]{Beck17} and \cite{Nes04}). To see this, as $g$ is convex, one has $\beta=0$. Choose $\overline{\mu}=0$, $\delta=\frac{\ell M}{m}$, and $\overline{\kappa}\in (0,1)$. Let $\kappa_n = \overline{\kappa}\frac{\nu_{n-1}-1}{\nu_n}$, where
\begin{equation*}
\nu_{-1}=\nu_0=1, \text{~and~} \nu_{n+1}=\frac{1+\sqrt{1+4 \nu_n^2}}{2},
\end{equation*}
and reset $\nu_{n-1}=\nu_n=1$ when $n=n_0,2n_0,3n_0,\dots$ for some integer $n_0$.
 In this case, it can be directly verified that $0\leq \kappa_n \leq \overline{\kappa} < 1$, and so, the requirement of our extrapolation parameter is satisfied. Also, it is worth noting that our proposed algorithm (Algorithm~\ref{algo:main}) allows one to perform extrapolation even when the smooth part of the numerator $f^{\mathfrak{s}} \equiv 0$ (as in the the motivating examples  \ref{r:ex_ssr} and \ref{r:ex_Sharpe} in the introduction).
\end{remark}

Next, we establish the subsequential convergence of Algorithm~\ref{algo:main}. To do this, we will need the following lemmas which will be used later on. The first lemma shows that our Assumption~\ref{a:g} on weak convexity  implies an important subgradient inequality. The second lemma is known as the decent lemma for differentiable function whose gradient is Lipschitz continuous.

\begin{lemma}[Subgradient inequality for weakly convex functions]
\label{l:subgradine}
Let $S$ be a nonempty closed convex subset of $\mathcal{H}$. Suppose that either $g$ is regular and weakly convex with modulus $\beta$ on $S$, or $g$ is weakly convex with modulus $\beta$ on an open convex set $O$ containing $S$. Then, for all $x, y\in S$ and $u \in \partial_L g(x)$,
\begin{equation*}
\scal{u}{y-x} \leq g(y) -g(x) +\frac{\beta}{2}\|y-x\|^2.
\end{equation*}
\end{lemma}
\begin{proof}
Let $x, y\in S$. By assumption, $G :=g+\iota_C+\frac{\beta}{2}\|\cdot\|^2$ is a convex function for $C =S$ or $C =O$. This implies that   $\partial G(x) =\partial_L G(x) =\partial_L(g+\iota_C)(x) +\beta x$, where the second equality is from Lemma~\ref{l:calculus}\ref{l:calculus_sum}. If $C =O$ is an open set containing $S$, then $\partial_L(g+\iota_C)(x) =\partial_L g(x)$. In the case where $C =S$, since $S$ is convex and $g$ is regular on $S$, Lemma~\ref{l:calculus}\ref{l:calculus_sum} also implies that $\partial_L(g+\iota_C)(x) =\partial_L g(x) +\partial_L \iota_S(x)$. Noting that $0\in \partial_L \iota_S(x)$, we deduce that, in both cases, $\partial_L g(x) +\beta x\subseteq \partial_L(g+\iota_C)(x) +\beta x =\partial G(x)$.

Now, let any $u\in \partial_L g(x)$. Then $u +\beta x\in \partial G(x)$, and so
\begin{align*}
\scal{u}{y-x} =\scal{u+\beta x}{y-x} +\scal{-\beta x}{y-x}
&\leq G(y) -G(x) -\beta \scal{x}{y-x} \\
&=\left(g(y) +\iota_C(y) +\frac{\beta}{2}\|y\|^2\right) -\left(g(x) +\iota_C(x) +\frac{\beta}{2}\|x\|^2\right) -\beta\scal{x}{y-x} \\
&= g(y) -g(x) +\frac{\beta}{2}\|y-x\|^2,
\end{align*}
which completes the proof.
\end{proof}

\begin{lemma}[Descent lemma]
\label{l:descent}
Let $h\colon \mathcal{H}\to \mathbb{R}$ be a differentiable function whose gradient is Lipschitz continuous with modulus $\ell$. Then, for all $x, y\in \mathcal{H}$,
\begin{equation*}
h(y)\leq h(x) +\scal{\nabla h(x)}{y-x} +\frac{\ell}{2}\|y-x\|^2.
\end{equation*}
\end{lemma}
\begin{proof}
This follows from \cite[Lemma 1.2.3]{Nes04}, see also \cite[Lemma~2.64(i)]{BC17}.
\end{proof}

We are now ready to state the subsequential convergence of Algorithm~\ref{algo:main}.

\begin{theorem}[Subsequential convergence]
\label{t:cvg}
Let $(x_n)_{n\in \mathbb{N}}$ be the sequence generated by Algorithm~\ref{algo:main}. Suppose that Assumptions~\ref{a:f} and \ref{a:g} hold, and that the set $S_0 :=\{x\in S: \frac{f(x)}{g(x)}\leq \frac{f(x_0)}{g(x_0)}\}$ is bounded. Then the following hold:
\begin{enumerate}
\item\label{t:cvg_decrease}
For all $n\in \mathbb{N}$, $x_n\in S\cap \dom f$ and
\begin{equation*}
\left(\frac{f(x_{n+1})}{g(x_{n+1})} +c\|x_{n+1}-x_n\|^2\right) +\alpha\|x_{n+1}-x_n\|^2\leq \frac{f(x_n)}{g(x_n)} +c\|x_n-x_{n-1}\|^2,
\end{equation*}
where
\begin{equation}\label{e:c&alpha}
\begin{cases}
c :=\frac{\ell\overline{\kappa}^2}{2m}+\frac{\overline{\mu}}{2\sqrt{mM}},\; \alpha :=\frac{\delta(1-\sqrt{\beta}\zeta)}{2M}-\frac{\overline{\mu}}{\sqrt{mM}}-\frac{\ell\overline{\kappa}^2}{2m} &\text{if \eqref{a:bound} holds},\\
c :=0,\quad \alpha :=\frac{\delta(1-\sqrt{\beta}\zeta)}{2M'} \text{~with~} M' :=\sup_{x\in S_0} g(x) &\text{otherwise}.
\end{cases}
\end{equation}
Consequently, the sequence $\left(\frac{f(x_n)}{g(x_n)}\right)_{n\in \mathbb{N}}$ is convergent.
\item\label{t:cvg_seq}
The sequence $(x_n)_{n\in \mathbb{N}}$ is bounded, asymptotically regular\footnote{A sequence $(x_n)$ is said to be \emph{asymptotically regular} if $\|x_{n+1}-x_n\|\to 0$ as $n\to +\infty$.}, and satisfies
$\sum_{n=0}^{+\infty} \|x_{n+1}-x_n\|^2 <+\infty.$
\item\label{t:cvg_crit}
If $\liminf_{n\to +\infty} \tau_n =\overline{\tau} >0$, then, for every cluster point $\overline{x}$ of $(x_n)_{n\in \mathbb{N}}$, it holds that $\overline{x}\in S\cap\dom f$, $\lim_{n\to +\infty} \frac{f(x_n)}{g(x_n)} =\frac{f(\overline{x})}{g(\overline{x})}$, and $\overline{x}$ is a lifted stationary point of $\eqref{e:prob}$.
\end{enumerate}
\end{theorem}
\begin{proof}
\ref{t:cvg_decrease}\&\ref{t:cvg_seq}: First, it is clear that, for all $n\in \mathbb{N}$,
$x_n\in S\cap \dom f$, and so
\begin{equation}\label{e:positive}
g(x_n) >0 \text{~~and~~} \theta_n =\frac{f(x_n)}{g(x_n)}\geq 0.
\end{equation}
We see that, for all $n\in \mathbb{N}$ and $x\in S$,
\begin{align*}
&\phantom{=\ \ } f(x) +\frac{1}{2\tau_n}\|x-v_n-\tau_n\theta_ng_n\|^2 +\frac{\ell}{2}\|x-u_n\|^2 \\
& = f^\mathfrak{n}(x) +f^\mathfrak{s}(x) +\frac{1}{2\tau_n}\|x-v_n-\tau_n\theta_ng_n\|^2 +\frac{\ell}{2}\|x-u_n\|^2 \\
& \geq f^\mathfrak{n}(x) +f^\mathfrak{s}(u_n) +\scal{\nabla f^\mathfrak{s}(u_n)}{x-u_n} +\frac{1}{2\tau_n}\|x-v_n-\tau_n\theta_ng_n\|^2 +\frac{\ell}{2}\|x-u_n\|^2 \\
& \geq f^\mathfrak{n}(x_{n+1}) +f^\mathfrak{s}(u_n) +\scal{\nabla f^\mathfrak{s}(u_n)}{x_{n+1}-u_n} +\frac{1}{2\tau_n}\|x_{n+1}-v_n-\tau_n\theta_ng_n\|^2 +\frac{\ell}{2}\|x_{n+1}-u_n\|^2\\
& \geq f^\mathfrak{n}(x_{n+1}) +f^\mathfrak{s}(x_{n+1}) -\frac{\ell}{2}\|x_{n+1}-u_n\|^2 +\frac{1}{2\tau_n}\|x_{n+1}-v_n-\tau_n\theta_ng_n\|^2 +\frac{\ell}{2}\|x_{n+1}-u_n\|^2 \\
& = f(x_{n+1})+\frac{1}{2\tau_n}\|x_{n+1}-v_n-\tau_n\theta_ng_n\|^2,
\end{align*}
where the first inequality follows from the convexity of $f^\mathfrak{s}$, the second inequality is from the definition of $x_{n+1}$ in Step~\ref{step:main} of the algorithm, and the last inequality follows from the fact that $f^\mathfrak{s}$ is a differentiable function whose gradient is Lipschitz continuous with modulus $\ell$ (Lemma~\ref{l:descent} with $h=f^{\mathfrak{s}}$, $y=x_{n+1}$ and
$x=u_n$).
Therefore, for all $n\in \mathbb{N}$ and $x\in S$,
\begin{equation}\label{e:fxfx+}
f(x)\geq f(x_{n+1}) +\frac{1}{2\tau_n}(\|x_{n+1}-v_n\|^2-\|x-v_n\|^2) -\theta_n\scal{g_n}{x_{n+1}-x} -\frac{\ell}{2}\|x-u_n\|^2.
\end{equation}
Letting $x =x_n$ and noting that $x_{n+1}-v_n =(x_{n+1}-x_n) -\mu_n(x_n-x_{n-1})$, $x_n-v_n =-\mu_n(x_n-x_{n-1})$, and $x_n-u_n =-\kappa_n(x_n-x_{n-1})$, we have
\begin{align*}
f(x_n) &\geq f(x_{n+1}) +\frac{1}{2\tau_n}\left(\|x_{n+1}-x_n\|^2 -2\mu_n\scal{x_{n+1}-x_n}{x_n-x_{n-1}}\right) -\theta_n\scal{g_n}{x_{n+1}-x_n} -\frac{\ell\kappa_n^2}{2}\|x_n-x_{n-1}\|^2.
\end{align*}
Next, let $\omega\in \RPP$. By Young's inequality,
\begin{equation*}
\scal{x_{n+1}-x_n}{x_n-x_{n-1}}\leq \frac{1}{2\omega}\|x_{n+1}-x_n\|^2 +\frac{\omega}{2}\|x_n-x_{n-1}\|^2.
\end{equation*}
Since $x_n, x_{n+1} \in S$ and $g_n \in \partial_L g(x_n)$, Lemma~\ref{l:subgradine} implies that
\begin{equation*}
\scal{g_n}{x_{n+1}-x_n} \leq g(x_{n+1})-g(x_n)+ \frac{\beta}{2}\|x_{n+1}-x_n\|^2.
\end{equation*}
Combining the three above inequalities yields
\begin{align*}
f(x_n) &\geq f(x_{n+1}) +\frac{1}{2}\left(\frac{1}{\tau_n}-\beta\theta_n-\frac{\mu_n}{\omega\tau_n}\right)\|x_{n+1}-x_n\|^2 +\theta_n(g(x_n)-g(x_{n+1})) -\frac{1}{2}\left(\ell\kappa_n^2+\frac{\omega\mu_n}{\tau_n}\right)\|x_n-x_{n-1}\|^2.
\end{align*}
Since $1/\tau_n\geq \max\{\sqrt{\beta}\theta_n/\zeta, \delta\}\geq \sqrt{\beta}\theta_n/\zeta$ (and so, $\frac{1}{\tau_n}-\beta\theta_n \geq \frac{1-\sqrt{\beta}\zeta}{\tau_n}$) and $\theta_n =f(x_n)/g(x_n)$, dividing $g(x_{n+1}) >0$ on both sides, it follows that
\begin{equation}\label{e:F-decrease}
\frac{f(x_n)}{g(x_n)} +\frac{1}{2g(x_{n+1})}\left(\ell\kappa_n^2+\frac{\omega\mu_n}{\tau_n}\right)\|x_n-x_{n-1}\|^2 \geq \\ \frac{f(x_{n+1})}{g(x_{n+1})} +\frac{1}{2g(x_{n+1})}\left(\frac{1-\sqrt{\beta}\zeta}{\tau_n}-\frac{\mu_n}{\omega\tau_n}\right)\|x_{n+1}-x_n\|^2.
\end{equation}
We now distinguish two following cases.

\emph{Case~1:} \eqref{a:bound} holds. Combining with $\kappa_n\leq \overline{\kappa}$, $\mu_n\leq \overline{\mu}\tau_n$, $1/\tau_n\geq \delta$, and choosing $\omega:=\sqrt{m/M}$, we derive from \eqref{e:F-decrease} that
\begin{equation*}
\frac{f(x_n)}{g(x_n)} +\frac{\ell\overline{\kappa}^2+\omega\overline{\mu}}{2m}\|x_n-x_{n-1}\|^2 \geq \frac{f(x_{n+1})}{g(x_{n+1})} +\left(\frac{\delta(1-\sqrt{\beta}\zeta)}{2M}-\frac{\overline{\mu}}{2M\omega}\right)\|x_{n+1}-x_n\|^2,
\end{equation*}
which means
\begin{equation*}
\frac{f(x_n)}{g(x_n)} +\left(\frac{\ell\overline{\kappa}^2}{2m}+\frac{\overline{\mu}}{2\sqrt{mM}}\right)\|x_n-x_{n-1}\|^2 \geq \frac{f(x_{n+1})}{g(x_{n+1})} +\left(\frac{\delta(1-\sqrt{\beta}\zeta)}{2M}-\frac{\overline{\mu}}{2\sqrt{mM}}\right)\|x_{n+1}-x_n\|^2.
\end{equation*}
Setting $F_n :=\frac{f(x_n)}{g(x_n)} +c\|x_n-x_{n-1}\|^2$, we deduce that
\begin{equation}\label{e:decrease}
F_{n+1} +\alpha\|x_{n+1}-x_n\|^2\leq F_n.
\end{equation}
From the choice of $\overline{\kappa}$, we have $\frac{\ell\overline{\kappa}^2}{2m} <\frac{\delta(1-\sqrt{\beta}\zeta)}{2M} -\frac{\overline{\mu}}{\sqrt{mM}}$. Thus, $\alpha >0$ and the sequence $(F_n)_{n\in \mathbb{N}}$ is nonincreasing. As $F_n$ is nonnegative, $(F_n)_{n\in \mathbb{N}}$ is a convergent sequence, say $F_n \to \overline{F}$. Furthermore, one also has from \eqref{e:decrease} that, for any positive integer $m$,
\begin{equation*}
\sum_{n=0}^m \alpha\|x_{n+1}-x_n\|^2 \leq \sum_{n=0}^m (F_n-F_{n+1}) =F_0-F_{m+1} \leq F_0.
\end{equation*}
It follows that
\begin{equation*}
\sum_{n=0}^{+\infty} \|x_{n+1}-x_n\|^2 <+\infty.
\end{equation*}
In particular, $x_{n+1}-x_n \to 0$ as $n\to +\infty$, and so
\begin{equation*}
\frac{f(x_n)}{g(x_n)} =F_n -c\|x_n-x_{n-1}\|^2 \to \overline{F} \quad\text{as~} n\to +\infty.
\end{equation*}
Next, to see the boundedness of $(x_n)_{n\in \mathbb{N}}$, observe that
\begin{equation*}
\frac{f(x_n)}{g(x_n)} \leq F_n \leq F_0=\frac{f(x_0)}{g(x_0)} +c\|x_{0}-x_{-1}\|^2=\frac{f(x_0)}{g(x_0)}.
\end{equation*}
So, $x_n \in S_0 =\{x \in S: \frac{f(x)}{g(x)} \leq \frac{f(x_0)}{g(x_0)}\}$, and hence $(x_n)_{n\in \mathbb{N}}$ is bounded by the assumption that $S_0$ is bounded.

\emph{Case~2}: \eqref{a:bound} does not hold. Then, by the construction of the algorithm,  $\overline{\mu} =\overline{\kappa} =0$, so $\mu_n =\kappa_n =0$ for all $n\in \mathbb{N}$ and \eqref{e:F-decrease} becomes
\begin{equation*}
\frac{f(x_n)}{g(x_n)} \geq \frac{f(x_{n+1})}{g(x_{n+1})} +\frac{1-\sqrt{\beta}\zeta}{2\tau_ng(x_{n+1})}\|x_{n+1}-x_n\|^2,
\end{equation*}
which implies that $(\theta_n)_{n\in \mathbb{N}}=(\frac{f(x_n)}{g(x_n)})_{n\in \mathbb{N}}$ is nonincreasing. As $(\theta_n)_{n\in \mathbb{N}}$ is bounded below, it is  convergent. Therefore, for all $n\in \mathbb{N}$, $x_n\in S_0 =\{x\in S: \frac{f(x)}{g(x)}\leq \frac{f(x_0)}{g(x_0)}\}$, and the sequence $(x_n)_{n\in \mathbb{N}}$ is thus bounded. Combining with the continuity of $g$ on $S$ and the boundedness of $S_0$, one has $\sup_{n\in \mathbb{N}} g(x_n) \leq M' =\sup_{x \in S_0} g(x) <+\infty$. Since $1/\tau_n\geq \delta$, it follows that
\begin{equation}\label{e:xx+}
\frac{f(x_{n+1})}{g(x_{n+1})} +\frac{\delta(1-\sqrt{\beta}\zeta)}{2M'}\|x_{n+1}-x_n\|^2 \leq \frac{f(x_n)}{g(x_n)}.
\end{equation}
The asymptotic regularity of $(x_n)_{n\in \mathbb{N}}$ follows from the convergence of $(\theta_n)_{n\in \mathbb{N}}$ and \eqref{e:xx+}. Also, telescoping \eqref{e:xx+} yields
\begin{equation*}
\sum_{n=0}^{+\infty} \|x_{n+1}-x_n\|^2 <+\infty.
\end{equation*}

\ref{t:cvg_crit}: Let $\overline{x}$ be any cluster point of $(x_n)_{n\in \mathbb{N}}$ and let $(x_{k_n})_{n\in \mathbb{N}}$ be a subsequence of $(x_n)_{n\in \mathbb{N}}$ such that $x_{k_n}\to \overline{x}$. Then $\overline{x}\in S$ and, by the asymptotic regularity, $x_{k_n-1}\to \overline{x}$ and also $u_{k_n-1}\to \overline{x}$ and $v_{k_n-1}\to \overline{x}$. We have from \eqref{e:fxfx+} that, for all $n\in \mathbb{N}$ and $x\in S$,
\begin{equation}\label{e:fxkn}
f(x)\geq f(x_{k_n}) -\frac{1}{2\tau_{k_n-1}}\|x-v_{k_n-1}\|^2 -\theta_{k_n-1}\scal{g_{k_n-1}}{x_{k_n}-x} -\frac{\ell}{2}\|x-u_{k_n-1}\|^2.
\end{equation}
Since $g$ is continuous on an open set containing $S$, we have $g(x_{k_n})\to g(\overline{x}) >0$ and, by \eqref{e:robustness} and passing to a subsequence if necessary, we may assume that $g_{k_n-1}\to \overline{g}\in \partial_Lg(\overline{x})$. Letting $x =\overline{x}$ and $n\to +\infty$ in \eqref{e:fxkn} and noting that $\liminf_{n\to +\infty} \tau_n =\overline{\tau} >0$, we get $\limsup_{n\to +\infty} f(x_{k_n})\leq f(\overline{x})$. This together with the lower semicontinuity of $f$ implies that $\lim_{n\to +\infty} f(x_{k_n}) =f(\overline{x})$. It then follows that
\begin{equation*}
\lim_{n\to +\infty} \frac{f(x_n)}{g(x_n)} =\lim_{n\to +\infty} \frac{f(x_{k_n})}{g(x_{k_n})} =\frac{f(\overline{x})}{g(\overline{x})}.
\end{equation*}
Now, letting $n\to +\infty$ in \eqref{e:fxkn}, one has, for all $x\in S$,
\begin{equation*}
f(x)-f(\overline{x})\geq -\frac{1}{2\overline{\tau}}\|x-\overline{x}\|^2 -\frac{f(\overline{x})}{g(\overline{x})}\scal{\overline{g}}{\overline{x}-x} -\frac{\ell}{2}\|x-\overline{x}\|^2,
\end{equation*}
or equivalently, for all $x \in S$,
\begin{equation*}
\varphi(x)\geq \varphi(\overline{x}),
\quad\text{where~} \varphi(x) :=f(x) +\left(\frac{1}{2\overline{\tau}}+\frac{\ell}{2}\right)\|x-\overline{x}\|^2-\frac{f(\overline{x})}{g(\overline{x})}\scal{\overline{g}}{x}.
\end{equation*}
We must have $0\in \partial_L(\varphi+\iota_S)(\overline{x})$, and so
$\frac{f(\overline{x})}{g(\overline{x})}\overline{g}\in \partial_L(f+\iota_S)(\overline{x})$.
In particular,  $\overline{x}\in S\cap\dom f$.
Since $\overline{g}\in \partial_Lg(\overline{x})$, we obtain that
\begin{equation*}
0\in g(\overline{x})\partial_L(f+\iota_S)(\overline{x})-f(\overline{x})\partial_Lg(\overline{x}),
\end{equation*}
and the proof is complete.
\end{proof}

Next, we consider the following assumption.

\begin{assumption}
\label{a:add}
$f^\mathfrak{n} =f^\mathfrak{l} +\iota_C$, where $C$ is a nonempty closed subset of $\mathcal{H}$ such that $C\cap S\neq \varnothing$ and one of the following is satisfied:
\begin{enumerate}[label =(\alph*)]
\item\label{a:fLip}
$f^\mathfrak{l}$ is locally Lipschitz continuous on $\mathcal{H}$ and $C =S =\mathcal{H}$;
\item\label{a:fLip&reg}
$f^\mathfrak{l}$ is locally Lipschitz continuous on an open set containing $S\cap \dom f$, both $f^\mathfrak{l}$ and $C\cap S$ are regular on $S\cap \dom f$, and $g$ is positive on an open set containing $S\cap \dom f$;
\item\label{a:fDiff}
$f^\mathfrak{l}$ is strictly differentiable on an open set containing $S\cap \dom f$ and $g$ is positive on an open set containing $S\cap \dom f$.
\end{enumerate}
\end{assumption}

All of our motivating examples in the introduction satisfy this assumption. Indeed, we note that convex sets and the unit sphere $C=\{x\in \mathbb{R}^N: \|x\|=1\}$ are regular, a continuous convex function on $\mathbb{R}^N$ is regular at any $x \in \mathbb{R}^N$, and $\iota_C$ is regular at any $x \in C$. It follows that  examples~\ref{r:ex_ssr} and \ref{r:ex_Sharpe} both satisfy Assumption~\ref{a:add}\ref{a:fLip}\&\ref{a:fLip&reg}, while example~\ref{r:rayleigh} satisfies Assumption~\ref{a:add}\ref{a:fLip&reg}\&\ref{a:fDiff}.

\begin{corollary}
\label{c:cvg-stationary}
Under the hypotheses of Theorem~\ref{t:cvg}, suppose further that $\liminf_{n\to +\infty} \tau_n =\overline{\tau} >0$, Assumption~\ref{a:add} holds, and $g$ is strictly differentiable on an open set containing $S\cap \dom f$. Then every cluster point $\overline{x}$ of $(x_n)_{n\in \mathbb{N}}$ is a stationary point of \eqref{e:prob}.
\end{corollary}
\begin{proof}
This follows from Theorem~\ref{t:cvg}\ref{t:cvg_crit} and Lemma~\ref{l:stationary}\ref{l:stationary_equi}.
\end{proof}

\section{A unified analysis framework and global convergence of e-PSG}
\label{s:KL}

In this section, we will prove that the global convergence of the whole sequence of $(x_n)_{n\in \mathbb{N}}$ generated by Algorithm~\ref{algo:main}, under the assumption that a suitable merit function satisfies the KL property.
To do this, we first establish a general framework for analyzing descent methods which is amenable for optimization method with multi-steps and inexact subproblems. As we will see later on, the proximal subgradient method with extrapolation which we proposed fits to this framework, and so, our desired global convergence result follows consequently.

Firstly, we fix some notation which will be used later on. Let $\mathcal{H}, \mathcal{K}$ be two finite-dimensional real Hilbert spaces. Let $h\colon \mathcal{K}\to \left(-\infty, +\infty\right]$ be a proper lower semicontinuous function, let $(x_n)_{n\in \mathbb{N}}$ and $(z_n)_{n\in \mathbb{N}}$ be respectively sequences in $\mathcal{H}$ and $\mathcal{K}$, $(\alpha_n)_{n\in \mathbb{N}}$ and $(\beta_n)_{n\in \mathbb{N}}$ sequences in $\RPP$,  $(\Delta_n)_{n\in \mathbb{N}}$ and  $(\varepsilon_n)_{n\in \mathbb{N}}$ sequences in $\RP$, and let $\underline{\imath}\leq \overline{\imath}$ be two (not necessarily positive) integers and $\lambda_i\in \RP$, $i\in I :=\{\underline{\imath}, \underline{\imath}+1, \dots, \overline{\imath}\}$, with $\sum_{i\in I} \lambda_i =1$. We set $\Delta_k =0$ for $k <0$ and consider the following conditions:
\begin{enumerate}
\renewcommand\theenumi{\rm (H\arabic{enumi})}
\renewcommand{\labelenumi}{\rm H\arabic{enumi}}
\item\label{a:decrease}
(\emph{Sufficient decrease condition}). For each $n\in \mathbb{N}$,
\begin{equation*}
h(z_{n+1}) +\alpha_n\Delta_n^2\leq h(z_n);
\end{equation*}
\item\label{a:error}
(\emph{Relative error condition}). For each $n\in \mathbb{N}$,
\begin{equation*}
\beta_{n}\dist(0,\partial_L h(z_{n}))\leq \sum_{i\in I} \lambda_i\Delta_{n-i} +\varepsilon_{n};
\end{equation*}
\item\label{a:continuity}
(\emph{Continuity condition}). There exist a subsequence $(z_{k_n})_{n\in \mathbb{N}}$ and $\widetilde{z}$ such that
\begin{equation*}
z_{k_n}\to \widetilde{z} \quad\text{and}\quad h(z_{k_n})\to h(\widetilde{z}) \quad\text{as~} n\to +\infty;
\end{equation*}
\item\label{a:parameter}
(\emph{Parameter condition}). It holds that
\begin{equation*}
\underline{\alpha} :=\inf_{n\in \mathbb{N}} \alpha_n >0,\quad
\underline{\gamma} :=\inf_{n\in \mathbb{N}} \alpha_n\beta_n >0,\quad\text{and}\quad
\sum_{n=1}^{+\infty} \varepsilon_n <+\infty;
\end{equation*}
\item\label{a:distance}
(\emph{Distance condition}). There exist $j\in \mathbb{Z}$ and $c\in \mathbb{R}$ such that, for all $n\in \mathbb{N}$,
\begin{equation*}
\|x_{n+1}-x_n\|\leq c\Delta_{n+j}.
\end{equation*}
\end{enumerate}

Next, we present a lemma which serves as a preparation for our abstract convergence result later on.
\begin{lemma}
\label{l:noKL}
Suppose that \ref{a:decrease} and \ref{a:continuity} hold. Let $\Omega$ be the set of cluster points of $(z_n)_{n\in \mathbb{N}}$ and set $\Omega_0 :=\{\overline{z}\in \Omega: h(z_n)\to h(\overline{z}) \text{~as~} n\to +\infty\}$. Then the following hold:
\begin{enumerate}
\item\label{l:noKL_f}
$\Omega_0 =\{\overline{z}\in \mathcal{K}: \exists z_{k_n}\to \overline{z} \text{~with~} h(z_{k_n})\to h(\overline{z}) \text{~as~} n\to +\infty\}\neq \varnothing$ and, for all $\overline{z}\in \Omega_0$,
\begin{equation*}
h(z_n)\downarrow h(\overline{z}) \text{~~as~} n\to +\infty.
\end{equation*}
\item\label{l:noKL_Delta}
If $\underline{\alpha} :=\inf_{n\in \mathbb{N}} \alpha_n >0$, then, for all $\overline{z}\in \Omega_0$,
\begin{equation*}
\sum_{n=0}^{+\infty} \Delta_n^2\leq  \frac{h(z_0)-h(\overline{z})}{\underline{\alpha}} <+\infty
\end{equation*}
and, consequently, $\Delta_n\to 0$ as $n\to +\infty$.
\item\label{l:noKL_crit}
If \ref{a:error} holds and $\underline{\delta} :=\inf_{n\in \mathbb{N}, i\in I} \alpha_{n-i}\beta_n^2 >0$, then, for all $n\geq \max\{0, \overline{\imath}\}$,
\begin{equation*}
\dist(0,\partial_L h(z_n))\leq \sqrt{\frac{h(z_{n-\overline{\imath}})-h(z_{n+1-\underline{\imath}})}{\underline{\delta}}} +\frac{\varepsilon_n}{\beta_n}.
\end{equation*}
If additionally $\lim_{n\to +\infty} \varepsilon_n/\beta_n =0$, then, for all $\overline{z}\in \Omega_0$,
\begin{equation*}
0\in \partial_L h(\overline{z}).
\end{equation*}
\end{enumerate}
\end{lemma}
\begin{proof}
\ref{l:noKL_f}: We first have from \ref{a:decrease} that $(h(z_n))_{n\in \mathbb{N}}$ is nondecreasing. Therefore, $(h(z_n))_{n\in \mathbb{N}}$ is convergent if and only if it has a converging subsequence. It follows that
\begin{equation*}
\Omega_0 =\{\overline{z}\in \mathcal{K}: \exists z_{k_n}\to \overline{z} \text{~with~} h(z_{k_n})\to h(\overline{z}) \text{~as~} n\to +\infty\}
\end{equation*}
and by \ref{a:continuity}, $\Omega_0\neq \varnothing$. The remaining statement follows from the definition of $\Omega_0$ and the monotonicity of $(h(z_n))_{n\in \mathbb{N}}$.

\ref{l:noKL_Delta}: Let $\overline{z}\in \Omega_0$. By \ref{a:decrease} and \ref{l:noKL_f},
\begin{equation*}
\sum_{n=0}^{+\infty} \alpha_n\Delta_n^2\leq \sum_{n=0}^{+\infty} (h(z_n)-h(z_{n+1})) =h(z_0)-h(\overline{z}).
\end{equation*}
Since $\underline{\alpha} =\inf_{n\in \mathbb{N}} \alpha_n >0$, it follows that
\begin{equation*}
\sum_{n=0}^{+\infty} \Delta_n^2\leq  \frac{h(z_0)-h(\overline{z})}{\underline{\alpha}} <+\infty,
\end{equation*}
and hence, $\Delta_n\to 0$ as $n\to +\infty$.

\ref{l:noKL_crit}: Assume that \ref{a:error} holds and $\underline{\delta} :=\inf_{n\in \mathbb{N}, i\in I} \alpha_{n-i}\beta_n^2 >0$. Let $n\geq \max\{0, \overline{\imath}\}$. Applying Cauchy--Schwarz inequality and using the fact that $\sum_{i\in I} \lambda_i^2\leq \sum_{i\in I} \lambda_i =1$, we have
\begin{equation*}
\left(\sum_{i\in I} \lambda_i\Delta_{n-i}\right)^2 \leq \left(\sum_{i\in I} \lambda_i^2\right) \left(\sum_{i\in I} \Delta_{n-i}^2\right) \leq \sum_{i\in I} \Delta_{n-i}^2.
\end{equation*}
Combining with \ref{a:error} and then with \ref{a:decrease} yields
\begin{equation*}
\beta_n\dist(0,\partial_L h(z_n))\leq \sqrt{\sum_{i\in I} \Delta_{n-i}^2} +\varepsilon_n\leq \sqrt{\sum_{i\in I} \frac{h(z_{n-i})-h(z_{n+1-i})}{\alpha_{n-i}}} +\varepsilon_n.
\end{equation*}
Since $\underline{\delta} =\inf_{n\in \mathbb{N}, i\in I} \alpha_{n-i}\beta_n^2 >0$, we derive that
\begin{align*}
\dist(0,\partial_L h(z_n))&\leq \sqrt{\sum_{i\in I} \frac{h(z_{n-i})-h(z_{n+1-i})}{\alpha_{n-i}\beta_n^2}} +\frac{\varepsilon_n}{\beta_n}\\
&\leq \sqrt{\sum_{i\in I} \frac{h(z_{n-i})-h(z_{n+1-i})}{\underline{\delta}}} +\frac{\varepsilon_n}{\beta_n}\\
&=\sqrt{\frac{h(z_{n-\overline{\imath}})-h(z_{n+1-\underline{\imath}})}{\underline{\delta}}} +\frac{\varepsilon_n}{\beta_n}.
\end{align*}
Finally, if $\lim_{n\to +\infty} \varepsilon_n/\beta_n =0$, then, noting from \ref{l:noKL_f} that $(h(z_n))_{n\in \mathbb{N}}$ is convergent, we get $\lim_{n\to +\infty} \dist(0,\partial_L h(z_n)) =0$. This shows that
$0\in \partial_L h(\overline{z})$ for all $\overline{z}\in \Omega_0$, which completes the proof.
\end{proof}

\begin{theorem}[Abstract convergence]
\label{t:abstract}
Suppose that \ref{a:decrease}, \ref{a:error}, \ref{a:continuity}, and \ref{a:parameter} hold and that the sequence $(z_n)_{n\in \mathbb{N}}$ is bounded. Let $\Omega$ be the set of cluster points of $(z_n)_{n\in \mathbb{N}}$ and suppose that $h$ is constant on $\Omega$ and satisfies the KL property at each point of $\Omega$. Set $\Omega_0 :=\{\overline{z}\in \Omega: h(z_n)\to h(\overline{z}) \text{~as~} n\to +\infty\}$ and $\overline{h} :=h(z)$ for $z\in \Omega_0$. Then the following hold:
\begin{enumerate}
\item\label{t:abstract_Delta}
The sequence $(\Delta_n)_{n\in \mathbb{N}}$ satisfies
\begin{equation*}
\sum_{n=0}^{+\infty} \Delta_n <+\infty.
\end{equation*}
\item\label{t:abstract_xn}
If \ref{a:distance} holds, then
$\sum_{n=0}^{+\infty} \|x_{n+1}-x_n\| <+\infty$,
and the sequence $(x_n)_{n\in \mathbb{N}}$ is convergent.
\item\label{t:abstract_crit}
If $\inf_{n\in \mathbb{N}} \beta_n >0$, then, for all $\overline{z}\in \Omega_0$,
\begin{equation*}
0\in \partial_L h(\overline{z}).
\end{equation*}
\item\label{t:abstract_linear}
Suppose further that $h$ satisfies the KL property at every point of $\Omega$ with an exponent of $\alpha\leq 1/2$, that $\underline{\imath}\leq 1$, and that
\begin{equation}\label{e:epsilon_n}
\delta :=\inf_{n\in \mathbb{N}, i\in I} \alpha_{n-i}\beta_n^2 >0
\quad\text{and}\quad
\frac{\varepsilon_n}{\beta_n} =O\left(\sqrt{h(z_{n-\overline{\imath}})-h(z_{n+1-\underline{\imath}})}\right) \text{~~as~} n\to +\infty.
\end{equation}
Then there exist $\gamma_1\in \RPP$ and $\rho\in \left(0,1\right)$ such that, for all $n\in \mathbb{N}$,
\begin{equation*}
h(z_n)-\overline{h}\leq \gamma_1\rho^n.
\end{equation*}
Moreover, if additionally \ref{a:distance} holds and $\sum_{k=n}^{+\infty} \varepsilon_k =O\left(\sqrt{h(z_{n-\overline{\imath}})-\overline{h}}\right)$ as $n\to +\infty$, then there exist $\overline{x}\in \mathcal{H}$ and $\gamma_2\in \RPP$ such that, for all $n\in \mathbb{N}$,
\begin{equation*}
\|x_n-\overline{x}\|\leq \gamma_2\rho^\frac{n}{2}.
\end{equation*}
\end{enumerate}
\end{theorem}
\begin{proof}
First, $\Omega_0\neq \varnothing$ due to Lemma~\ref{l:noKL}\ref{l:noKL_f}. Let $\overline{z}\in \Omega_0$. Again by Lemma~\ref{l:noKL}\ref{l:noKL_f},
\begin{equation*}
h(z_n)\downarrow \overline{h} =h(\overline{z}) \text{~~as~} n\to +\infty.
\end{equation*}

\ref{t:abstract_Delta}: Noting that, for all $n\in \mathbb{N}$, $h(z_n)\geq h(\overline{z})$, we distinguish the following two cases.

\emph{Case~1:} There exists $n_1\in \mathbb{N}$ such that $h(z_{n_1}) =h(\overline{z})$. Then, since $(h(z_n))_{n\in \mathbb{N}}$ is nondecreasing, $h(z_n) =h(\overline{z})$ for all $n\geq n_1$. It follows from \ref{a:decrease} that $\Delta_n =0$ for all $n\geq n_1$, so $\sum_{n=0}^{+\infty} \Delta_n <+\infty$.

\emph{Case~2:} For all $n\in \mathbb{N}$, $h(z_n) >h(\overline{z})$. We derive from Lemma~\ref{l:uniformKL} that there exist $\eta \in \left(0,+\infty\right]$, $\varphi\in \Phi_\eta$, and $n_0\in \mathbb{N}$ such that, for all $n\geq n_0$,
\begin{equation}\label{e:KLine'}
\varphi'(h(z_n)-h(\overline{z}))\dist(0,\partial_L h(z_n))\geq 1.
\end{equation}
Setting $r_n :=h(z_n)-h(\overline{z})\downarrow 0$, by combining with \ref{a:decrease}, \ref{a:error}, \ref{a:parameter}, and the concavity of $\varphi$, it follows that, for all $n\geq n_0$,
\begin{align*}
\Delta_n^2&\leq \frac{1}{\alpha_n}(h(z_n)-h(z_{n+1}))\varphi'(h(z_n)-h(\overline{z}))\dist(0,\partial_L h(z_n))\\
&\leq \frac{1}{\alpha_n\beta_n}\Big(\varphi(r_n) -\varphi(r_{n+1})\Big) \Big(\sum_{i\in I} \lambda_i\Delta_{n-i} +\varepsilon_n\Big)\\
&\leq \frac{1}{\underline{\gamma}}\Big(\varphi(r_n) -\varphi(r_{n+1})\Big) \Big(\sum_{i\in I} \lambda_i\Delta_{n-i} +\varepsilon_n\Big).
\end{align*}
Using the inequality of arithmetic and geometric means (AM-GM) gives us that, for all $n\geq n_0$,
\begin{equation*}
2\Delta_n\leq \frac{1}{\underline{\gamma}}\Big(\varphi(r_n) -\varphi(r_{n+1})\Big) +\Big(\sum_{i\in I} \lambda_i\Delta_{n-i} +\varepsilon_n\Big).
\end{equation*}
Since this inequality holds for all $n\geq n_0$, we derive that, for all $m\geq n\geq \max\{n_0, \overline{\imath}\}$,
\begin{equation}\label{e:2sum}
2\sum_{k=n}^m \Delta_k\leq \frac{1}{\underline{\gamma}}\Big(\varphi(r_n) -\varphi(r_{m+1})\Big) +\sum_{k=n}^m\sum_{i\in I} \lambda_i\Delta_{k-i} +\sum_{k=n}^m  \varepsilon_k.
\end{equation}
We have that
\begin{equation*}
\sum_{k=n}^m\sum_{i\in I} \lambda_i\Delta_{k-i} =\sum_{i\in I} \lambda_i\sum_{k=n}^m \Delta_{k-i} =\sum_{i\in I} \lambda_i\sum_{k=n-i}^{m-i} \Delta_k \leq \sum_{i\in I} \lambda_i\sum_{k=n-\overline{\imath}}^{m-\underline{\imath}} \Delta_k =\sum_{k=n-\overline{\imath}}^{m-\underline{\imath}} \Delta_k,
\end{equation*}
using the fact that $\Delta_k\geq 0$ for all $k\in \mathbb{Z}$ and that $\sum_{i\in I} \lambda_i =1$. Now, by adopting the convention that a summation is zero when the starting index is larger than the ending index,
\begin{equation*}
\sum_{k=n-\overline{\imath}}^{m-\underline{\imath}} \Delta_k\leq \sum_{k=n}^m \Delta_k +\sum_{k=n-\overline{\imath}}^{n-1} \Delta_k +\sum_{k=m+1}^{m-\underline{\imath}} \Delta_k
=\sum_{k=n}^m \Delta_k +\sum_{i=1}^{\overline{\imath}} \Delta_{n-i} +\sum_{i=\underline{\imath}}^{-1} \Delta_{m-i}.
\end{equation*}
We continue \eqref{e:2sum} as
\begin{equation*}
\sum_{k=n}^m \Delta_k\leq \frac{1}{\underline{\gamma}}\Big(\varphi(r_n) -\varphi(r_{m+1})\Big) +\sum_{i=1}^{\overline{\imath}} \Delta_{n-i} +\sum_{i=\underline{\imath}}^{-1} \Delta_{m-i} +\sum_{k=n}^m \varepsilon_k.
\end{equation*}
Letting $m\to +\infty$ and noting from Lemma~\ref{l:noKL}\ref{l:noKL_Delta} that $\Delta_m\to 0$, we obtain
\begin{equation}\label{e:sumDelta}
\sum_{k=n}^{+\infty} \Delta_k\leq \frac{1}{\underline{\gamma}}\varphi(r_n) +\sum_{i=1}^{\overline{\imath}} \Delta_{n-i} +\sum_{k=n}^{+\infty} \varepsilon_k <+\infty,
\end{equation}
which yields $\sum_{n=0}^{+\infty} \Delta_n <+\infty$.

\ref{t:abstract_xn}: This follows from \ref{t:abstract_Delta} and \ref{a:distance}.

\ref{t:abstract_crit}: As $\inf_{n\in \mathbb{N}} \beta_n >0$, noting that $\inf_{n\in \mathbb{N}} \alpha_n >0$ and $\lim_{n\to +\infty} \varepsilon_n =0$, we have $\inf_{n\in \mathbb{N}, i\in I} \alpha_{n-i}\beta_n^2 >0$ and $\lim_{n\to +\infty} \varepsilon_n/\beta_n =0$. Therefore, the conclusion of this part follows from Lemma~\ref{l:noKL}\ref{l:noKL_crit}.

\ref{t:abstract_linear}: Using Lemma~\ref{l:noKL}\ref{l:noKL_crit} and \eqref{e:epsilon_n}, and by increasing $n_0$ if necessary, we find $c_1\in \RPP$ such that, for all $n\geq n_0$,
\begin{equation*}
\dist(0,\partial_L h(z_n))\leq \sqrt{\frac{h(z_{n-\overline{\imath}})-h(z_{n+1-\underline{\imath}})}{\delta}} +\frac{\varepsilon_n}{\beta_n}\leq c_1\sqrt{h(z_{n-\overline{\imath}})-h(z_{n+1-\underline{\imath}})}.
\end{equation*}
Combining with \eqref{e:KLine'} yields
\begin{equation*}
1\leq c_1\varphi'(h(z_n)-h(\overline{z}))\sqrt{h(z_{n-\overline{\imath}})-h(z_{n+1-\underline{\imath}})}.
\end{equation*}
Since $r_n =h(z_n)-h(\overline{z})$, it follows that
\begin{equation*}
1\leq c_1^2 \, [\varphi'(r_n)]^2 \, (r_{n-\overline{\imath}}-r_{n+1-\underline{\imath}}).
\end{equation*}
As $\varphi$ can be chosen as $\varphi(s) =\gamma s^{1-\alpha}$ for some $\gamma\in \RPP$, there exists $c_2\in \RPP$ such that, for all $n\geq n_0$,
\begin{equation*}
c_2r_n^{2\alpha}\leq r_{n-\overline{\imath}}-r_{n+1-\underline{\imath}}.
\end{equation*}
Since $r_n\downarrow 0$, $2\alpha\leq 1$, and $1-\underline{\imath}\geq 0$, by increasing $n_0$ if necessary, it holds that, for all $n\geq n_0$,
$r_n^{2\alpha}\geq r_n\geq r_{n+1-\underline{\imath}}$.
We deduce that, for all $n\geq n_0$,
\begin{equation*}
r_{n+1-\underline{\imath}}\leq \frac{1}{c_2+1}r_{n-\overline{\imath}},
\end{equation*}
and hence, there exist $\gamma_1\in \RPP$ and $\rho\in \left(0,1\right)$ such that, for all $n\in \mathbb{N}$,
$0\leq r_n =h(z_n)-h(\overline{z})\leq \gamma_1\rho^n$.

Now, it follows from Cauchy--Schwarz inequality, \ref{a:decrease}, and \ref{a:parameter} that
\begin{align*}
\left(\sum_{i=1}^{\overline{\imath}} \Delta_{n-i}\right)^2\leq \overline{\imath}\sum_{i=1}^{\overline{\imath}} \Delta_{n-i}^2
&\leq \overline{\imath}\sum_{i=1}^{\overline{\imath}} \frac{h(z_{n-i})-h(z_{n+1-i})}{\underline{\alpha}}\\
&\leq \frac{\overline{\imath}}{\underline{\alpha}}\max\{0, h(z_{n-\overline{\imath}})-h(z_n)\}\\
&\leq \frac{\overline{\imath}}{\underline{\alpha}}r_{n-\overline{\imath}}.
\end{align*}
Combining with \eqref{e:sumDelta} gives, for all $n\geq n_0$,
\begin{equation*}
\sum_{k=n}^{+\infty} \Delta_k\leq \frac{1}{\underline{\gamma}}\varphi(r_n) +\sqrt{\frac{\overline{\imath}}{\underline{\alpha}}r_{n-\overline{\imath}}} +\sum_{k=n}^{+\infty} \varepsilon_k.
\end{equation*}
As $\sum_{k=n}^{+\infty} \varepsilon_k =O\left(\sqrt{h(z_{n-\overline{\imath}})-h(\overline{z})}\right) =O(\sqrt{r_{n-\overline{\imath}}})$ as $n\to +\infty$ and $\varphi(r_n) =\gamma r_n^{1-\alpha}\leq \gamma r_{n-\overline{\imath}}^{1-\alpha}\leq \gamma\sqrt{r_{n-\overline{\imath}}}$ for all $n$ large enough, by increasing $n_0$ if necessary, there exists $c_3\in \RPP$ such that, for all $n\geq n_0$,
\begin{equation*}
\sum_{k=n}^{+\infty} \Delta_k\leq c_3\sqrt{r_{n-\overline{\imath}}}\leq c_3\sqrt{\gamma_1}\rho^{\frac{n-\overline{\imath}}{2}}.
\end{equation*}
Since \ref{a:distance} holds, \ref{t:abstract_xn} implies that $(x_n)_{n\in \mathbb{N}}$ is convergent to some $\overline{x}\in \mathcal{H}$. Then, for all $n\in \mathbb{N}$,
\begin{equation*}
\|x_n-\overline{x}\|\leq \sum_{k=n}^{+\infty} \|x_{k+1}-x_k\|\leq c\sum_{k=n}^{+\infty} \Delta_k
\end{equation*}
and the conclusion follows.
\end{proof}

\begin{remark}[Parameter conditions]
\label{r:parameters}
In view of \ref{a:parameter} and as shown in the proof of Theorem~\ref{t:abstract}\ref{t:cvg_crit}, if $\inf_{n\in \mathbb{N}} \beta_n >0$, then the conditions $\inf_{n\in \mathbb{N}, i\in I} \alpha_{n-i}\beta_n^2 >0$ and $\lim_{n\to +\infty} \varepsilon_n/\beta_n =0$ in Lemma~\ref{l:noKL}\ref{l:noKL_crit} are guaranteed. If additionally $\varepsilon_n =O(h(z_{n-\overline{\imath}})-h(z_{n+1-\underline{\imath}}))$ as $n\to +\infty$, then the parameter conditions
\begin{equation*}
\frac{\varepsilon_n}{\beta_n} =O\left(\sqrt{h(z_{n-\overline{\imath}})-h(z_{n+1-\underline{\imath}})}\right) \quad\text{and}\quad \sum_{k=n}^{+\infty} \varepsilon_k =O\left(\sqrt{h(z_{n-\overline{\imath}})-h(\overline{z})}\right) \text{~~as~} n\to +\infty.
\end{equation*}
in Theorem~\ref{t:abstract}\ref{t:abstract_linear} are also satisfied. Indeed, since $h(z_n)\downarrow h(\overline{z})$, we have $h(z_{n-\overline{\imath}})-h(z_{n+1-\underline{\imath}})\to 0$ as $n\to +\infty$, and so, for all $n$ large enough, $h(z_{n-\overline{\imath}})-h(z_{n+1-\underline{\imath}})\leq \sqrt{h(z_{n-\overline{\imath}})-h(z_{n+1-\underline{\imath}})}$. It follows that $\varepsilon_n =O(\sqrt{h(z_{n-\overline{\imath}})-h(z_{n+1-\underline{\imath}})})$ and, since $\inf_{n\in \mathbb{N}} \beta_n >0$, $\varepsilon_n/\beta_n =O(\sqrt{h(z_{n-\overline{\imath}})-h(z_{n+1-\underline{\imath}})})$ as $n\to +\infty$. Now, we note that
\begin{equation*}
\sum_{k=n}^{+\infty} (h(z_{k-\overline{\imath}})-h(z_{k+1-\underline{\imath}})) =\sum_{i=\underline{\imath}}^{\overline{\imath}} (h(z_{n-i})-h(\overline{z}))\leq (\overline{\imath}-\underline{\imath}+1)(h(z_{n-\overline{\imath}})-h(\overline{z})),
\end{equation*}
which implies that $\sum_{k=n}^{+\infty} \varepsilon_k =O(h(z_{n-\overline{\imath}})-h(\overline{z}))$, and so $\sum_{k=n}^{+\infty} \varepsilon_k =O\left(\sqrt{h(z_{n-\overline{\imath}})-h(\overline{z})}\right)$ as $n\to +\infty$.
\end{remark}

\begin{remark}[Comparison to the existing literature]
The general framework \ref{a:decrease}--\ref{a:distance} extends various convergence conditions for exact and inexact descent methods in the literature. Specifically, in \cite{ABS13,BST14}, the authors proposed conditions that satisfied \ref{a:decrease}--\ref{a:distance} with $\mathcal{K} =\mathcal{H} =\mathbb{R}^N$, $z_n =x_n$, $\Delta_n =\|x_{n+1}-x_n\|^2$, $\alpha_n \equiv a$, $\beta_n \equiv 1/b$, $\varepsilon_n \equiv 0$, $I =\{1\}$, and $\lambda_1 =1$. These conditions were then generalized in \cite{FGP14} to flexible parameters and real Hilbert spaces. In the finite-dimensional setting, the conditions in \cite{FGP14} fulfill  \ref{a:decrease}--\ref{a:distance} with $\mathcal{K} =\mathcal{H}$, $z_n =x_n$, $\Delta_n =\|x_{n+1}-x_n\|^2$, $I =\{1\}$, and $\lambda_1 =1$.

The framework \ref{a:decrease}--\ref{a:distance} also holds in the case of \cite[Proposition~4]{BP16} with $\mathcal{K} =\mathcal{H} =\mathbb{R}^N$, $z_n =x_n$, $\Delta_n =\|x_{n+2}-x_{n+1}\|^2$, $\alpha_n \equiv a$, $\beta_n \equiv 1/b$, $\varepsilon_n \equiv 0$, $I =\{1\}$, and $\lambda_1 =1$. Here, $\Delta_n$ is shifted one step forward comparing to the two aforementioned studies. This difference makes the relative error condition explicit; see \cite[Section~2.4]{Noll13} for a discussion.

In \cite{OCBP14}, the authors provided a framework for convergence analysis of iPiano, a proximal gradient algorithm with extrapolation. In turn, their conditions satisfied \ref{a:decrease}--\ref{a:distance} with $\mathcal{K} =\mathcal{H}^2$, $z_n =(x_n,x_{n-1})$, $\Delta_n =\|x_n-x_{n-1}\|^2$, $\alpha_n \equiv a$, $\beta_n \equiv 1/b$, $\varepsilon_n \equiv 0$, $I =\{0,1\}$, and $\lambda_0 =\lambda_1 =1/2$. Recently, these conditions have been extended in \cite{Och19} with $\mathcal{H} =\mathbb{R}^N$, $\mathcal{K} =\mathbb{R}^{N+P}$ and $z_n =(x_n,u_n)$. It is worth noting that the finite index set $I$ of integers in \cite{Och19} can always be written as $I =\{\underline{\imath}, \underline{\imath}+1, \dots, \overline{\imath}\}$ for $\underline{\imath}\leq \overline{\imath}$. To get the global convergence of $(x_n)_{n\in \mathbb{N}}$, \cite[Theorem~10]{Och19} not only needs \ref{a:distance} as our Theorem~\ref{t:abstract} but also requires that $h$ is bounded from below and that, for any converging subsequence $(z_{k_n})_{n\in \mathbb{N}}$ of $(z_n)_{n\in \mathbb{N}}$,
\begin{equation*}
z_{k_n}\to \widetilde{z} \quad\text{and}\quad h(z_{k_n})\to h(\widetilde{z}) \quad\text{as~} n\to +\infty,
\end{equation*}
which implies that $h$ is constant on $\Omega$. We also note that linear convergence of $(x_n)_{n\in \mathbb{N}}$ has been not investigated in the framework of \cite{Och19,OCBP14}.
\end{remark}

Next, we show that the full sequence generated by Algorithm~\ref{algo:main} is globally convergent by further assuming that  a suitable merit function is a KL function. We note that, as we will see later in Remark \ref{r:KLfunctions}, this assumption is automatically fulfilled if $f$ and $g$ are both semi-algebraic functions and $S$ is a semi-algebraic set, which, in particular, holds for all the motivating examples mentioned before.

\begin{theorem}[Global convergence]
\label{t:cvgKL}
Let $\lim\inf_{n \to \infty} \tau_n =\overline{\tau} >0$ and let $(x_n)_{n\in \mathbb{N}}$ be the sequence generated by Algorithm~\ref{algo:main}. Suppose that Assumptions~\ref{a:f}, \ref{a:g}, and \ref{a:add} hold, that $g$ is continuously differentiable on an open set containing $S\cap \dom f$, that, for $c$ given in \eqref{e:c&alpha},
\begin{equation*}
h(x,y) :=\frac{f(x)}{g(x)} +\iota_S(x) +c\|x-y\|^2
\end{equation*}
satisfies the KL property at $(\overline{x},\overline{x})$ for all $\overline{x}\in S\cap \dom f$, and that the set $\{x\in S: \frac{f(x)}{g(x)}\leq \frac{f(x_0)}{g(x_0)}\}$ is bounded. Suppose that there exist $\varepsilon, \ell_g\in \RPP$ satisfying
\begin{equation*}
\forall x,y\in S\cap \dom f,\quad 
\|x-y\|\leq \varepsilon \implies \|\nabla g(x)-\nabla g(y))\|\leq \ell_g\|x-y\|.
\end{equation*}
Then
$\sum_{n=0}^{+\infty} \|x_{n+1}-x_n\| <+\infty$,
and the sequence $(x_n)_{n\in \mathbb{N}}$ converges to a stationary point of \eqref{e:prob}. Moreover, if $h$ satisfies the KL property with an exponent of $\alpha\leq 1/2$ at $(\overline{x},\overline{x})$ for all $\overline{x}\in S\cap \dom f$, then the convergence rate of $(x_n)_{n\in \mathbb{N}}$ and $(h(x_{n+1},x_n))_{n\in \mathbb{N}}$ is linear in the sense that there exist $\gamma_1, \gamma_2\in \RPP$ and $\rho\in \left(0,1\right)$ such that, for all $n\in \mathbb{N}$,
\begin{equation*}
|h(x_{n+1},x_n) -h(x_\infty,x_\infty)|\leq \gamma_1\rho^n \quad\text{and}\quad \|x_n-x_\infty\|\leq \gamma_2\rho^\frac{n}{2}.
\end{equation*}
\end{theorem}
\begin{proof}
Let $z_n=(x_{n+1},x_n)$.
Let $\Omega$ be the set of cluster points of $(z_n)_{n\in \mathbb{N}}$. Theorem~\ref{t:cvg} asserts that the sequence $(z_n)_{n\in \mathbb{N}}$ is in $(S\cap \dom f) \times (S\cap \dom f)$, bounded, and asymptotically regular. Moreover, 
for all $n\in \mathbb{N}$,
\begin{equation}\label{e:decrease'}
h(z_{n+1})+ \alpha\|x_{n+1}-x_n\|^2\leq h(z_n) \quad\text{with $\alpha >0$ given in \eqref{e:c&alpha}}. 
\end{equation}
By combining with Corollary~\ref{c:cvg-stationary}, for every $\overline{z}\in \Omega$, one has $\overline{z}=(\overline{x},\overline{x})$ with $\overline{x}\in S\cap\dom f$ a stationary point of \eqref{e:prob} and
\begin{equation*}
\theta_n=\frac{f(x_n)}{g(x_n)}\to \frac{f(\overline{x})}{g(\overline{x})} \quad\text{as~} n\to +\infty.
\end{equation*}
In particular, $h(z_n)=h(x_{n+1},x_n)=\frac{f(x_{n+1})}{g(x_{n+1})}+c\, \|x_{n+1}-x_n\|^2 \rightarrow \frac{f(\overline{x})}{g(\overline{x})}$ as $n \rightarrow +\infty$.

From Step~\ref{step:main} of Algorithm~\ref{algo:main} and noting that $g_n =\nabla g(x_n)$, we have for all $n\in \mathbb{N}$,
\begin{equation*}
0 \in \partial_L(f^\mathfrak{n}+\iota_S)(x_{n+1}) + \nabla f^\mathfrak{s}(u_n) +\frac{1}{\tau_n}(x_{n+1}-v_n-\tau_n\theta_n\nabla g(x_n))+ \ell(x_{n+1}-u_n),
\end{equation*}
which combined with $\partial_L(f+\iota_S) =\nabla f^\mathfrak{s} +\partial_L(f^\mathfrak{n}+\iota_S)$ yields
\begin{align*}
\hat{x}_{n+1} &:=\nabla f^\mathfrak{s}(x_{n+1}) -\nabla f^\mathfrak{s}(u_n) -\ell(x_{n+1}-u_n) -\frac{1}{\tau_n}(x_{n+1}-v_n) +\theta_n\nabla g(x_n) \\
&\phantom{:}\in \partial_L(f+\iota_S)(x_{n+1}).
\end{align*}
Since $f$ and $S$ are regular at $x_n$, $g$ is continuously differentiable at $x_n$, and $g(x_n) >0$, it holds that
\begin{align*}
\partial_Lh(z_n) 
&=\left\{\partial_L \left(\frac{f}{g}+\iota_S\right)(x_{n+1}) + 2c(x_{n+1}-x_n) \right\} \times \{ 2c(x_n-x_{n+1})\}\\
&=\left\{\frac{g(x_{n+1})\partial_L(f+\iota_S)(x_{n+1})-f(x_{n+1})\nabla g(x_{n+1})}{g(x_{n+1})^2}+ 2c(x_{n+1}-x_n)\right\} \times \{ 2c(x_n-x_{n+1})\}\\
&=\left\{\frac{\partial_L(f+\iota_S)(x_{n+1})-\theta_{n+1}\nabla g(x_{n+1})}{g(x_{n+1})}+ 2c(x_{n+1}-x_n)\right\} \times \{ 2c(x_n-x_{n+1})\},
\end{align*}
where the second equality follows from Lemma~\ref{l:stationary}\ref{l:stationary_equi}. Therefore, we have  $\{x_n^* +2c(x_{n+1}-x_n)\} \times\{ 2c(x_n-x_{n+1})\} \in \partial_Lh(z_n)$ with
\begin{equation*}
x_n^* := \frac{\hat{x}_{n+1} -\theta_{n+1}\nabla g(x_{n+1})}{g(x_{n+1})}.
\end{equation*}

Note that $\tau_n \leq 1/\max\{\sqrt{\beta}\theta_n/\zeta, \delta\} \leq \frac{1}{\delta}$, so $\mu_n \leq \overline{\mu}\tau_n \leq \frac{\overline{\mu}}{\delta}$.
Next, we see that, for all $n\in \mathbb{N}$,
\begin{align*}
\|x_{n+1}-v_n\| &\leq \|x_{n+1}-x_n\| +\mu_n\|x_n-x_{n-1}\|\leq \|x_{n+1}-x_n\| + \frac{\overline{\mu}}{\delta} \|x_n-x_{n-1}\|, \\
\|x_{n+1}-u_n\| &\leq \|x_{n+1}-x_n\| +\kappa_n\|x_n-x_{n-1}\|\leq \|x_{n+1}-x_n\| +\overline{\kappa}\|x_n-x_{n-1}\|,
\end{align*}
and by the Lipschitz continuity of $\nabla f^\mathfrak{s}$,
\begin{equation*}
\|\nabla f^\mathfrak{s}(x_{n+1})-\nabla f^\mathfrak{s}(u_n)\|\leq \ell\|x_{n+1}-u_n\|\leq \ell\|x_{n+1}-x_n\| +\ell\overline{\kappa}\|x_n-x_{n-1}\|.
\end{equation*}
Since $(x_n)_{n\in \mathbb{N}}$ is bounded, the continuity of $\nabla g$ implies that $(\nabla g(x_n))_{n\in \mathbb{N}}$ is also bounded. There thus exists $\mu\in \RPP$ such that, for all $n\in \mathbb{N}$,
$\|\nabla g(x_n)\|\leq \mu$.
Since $\liminf_{n\to +\infty} \tau_n =\overline{\tau} >0$ and $\lim_{n\to +\infty} \|x_{n+1}-x_n\| =0$, there exists $n_0\in \mathbb{N}$ such that, for all $n \geq n_0$,
\begin{equation*}
\tau_n \geq \overline{\tau}/2 \quad\text{and}\quad \|x_{n+1}-x_n\|\leq \varepsilon.
\end{equation*}
Now, from the definition of $h(z_n)$, we see that
\begin{align*}
\theta_n\nabla g(x_n)-\theta_{n+1}\nabla g(x_{n+1})
&= \theta_n(\nabla g(x_n)-\nabla g(x_{n+1})) -c\|x_n-x_{n-1}\|^2\nabla g(x_{n+1}) \\
&\quad +c\|x_{n+1}-x_n\|^2\nabla g(x_{n+1}) +(h(z_{n-1})-h(z_{n}))\nabla g(x_{n+1})
\end{align*}
and by the Lipschitz-type continuity of $\nabla g$ and the boundedness of $(\nabla g(x_n))$, for all $n\geq n_0$,
\begin{align*}
\|\theta_n\nabla g(x_n)-\theta_{n+1}\nabla g(x_{n+1})\|
&\leq \ell_g\theta_n\|x_{n+1}-x_n\| +c\varepsilon\mu\|x_n-x_{n-1}\| \\
&\qquad +c\varepsilon\mu\|x_{n+1}-x_n\| +\mu(h(z_{n-1})-h(z_{n})).
\end{align*}
Altogether, it follows from the definition of $x_n^*$ that, for all $n\geq n_0$,
\begin{align*}
\|\hat{x}_{n+1} -\theta_{n+1}\nabla g(x_{n+1})\|
&\leq \|\nabla f^\mathfrak{s}(x_{n+1})-\nabla f^\mathfrak{s}(u_n)\| +\ell\|x_{n+1}-u_n\| +\frac{1}{\tau_n}\|x_{n+1}-v_n\| \\
&\qquad +\|\theta_n\nabla g(x_n)-\theta_{n+1}\nabla g(x_{n+1})\| \\
&\leq 2\ell\|x_{n+1}-x_n\| +2\ell\overline{\kappa}\|x_n-x_{n-1}\| + \frac{2}{\overline{\tau}}(\|x_{n+1}-x_n\| + \frac{\overline{\mu}}{\delta}\|x_n-x_{n-1}\|) \\
&\qquad +(\ell_g\theta_n+c\varepsilon\mu)\|x_{n+1}-x_n\| +c\varepsilon\mu\|x_n-x_{n-1}\| +\mu(h(z_{n-1})-h(z_{n})).
\end{align*}
Noting that $(\theta_n)_{n\in \mathbb{N}}$ is convergent and hence bounded and recalling that $g(x)\geq m>0$ for all $x\in S\cap \dom f$, we find $K\in \RPP$ such that, for all $n\geq n_0$,
\begin{align*}
\|x_n^*\| &=\frac{\|\hat{x}_{n+1} -\theta_{n+1}\nabla g(x_{n+1})\|}{|g(x_{n+1})|} \\
&\leq K\left(\|x_{n+1}-x_n\| +\|x_n-x_{n-1}\| + (h(z_{n-1})-h(z_{n}))\right).
\end{align*}
We deduce that there exists $K_1\in \RPP$ such that, for all $n\geq n_0$,
\begin{align*}
\dist(0,\partial_L h(z_n)) &\leq \sqrt{\|x_n^* + 2c(x_{n+1}-x_n) \|^2 +4c^2 \|x_n-x_{n+1}\|^2} \\
& \leq  \sqrt{2\|x_n^*\|^2 + 8c^2 \|x_{n+1}-x_n\|^2 +4c^2 \|x_n-x_{n+1}\|^2}\\
&\leq  K_1\left(\|x_{n+1}-x_n\| +\|x_n-x_{n-1}\| +h(z_{n-1})-h(z_{n})\right),
\end{align*}
where the second inequality is from the elementary inequality that $\|a+b\|^2\leq 2\|a\|^2+2\|b\|^2$. Now,  by applying Theorem~\ref{t:abstract} and Remark~\ref{r:parameters} with $I =\{0,1\}$, $\lambda_0 =\lambda_1 =1/2$, $\Delta_n =2K_1\|x_{n+1}-x_n\|$, $\alpha_n \equiv \frac{\alpha}{4K_1^2}>0$, $\beta_n \equiv 1$, and $\varepsilon_n =K_1(h(z_{n-1})-h(z_{n})) \leq K_1(h(z_{n-1})-h(z_{n+1}))$, we get the conclusion.
\end{proof}

\begin{remark}
\label{r:KLfunctions}
In Theorem~\ref{t:cvgKL}, we impose the assumption that the merit function $h(x,y) =\frac{f(x)}{g(x)} +\iota_S(x) +c\|x-y\|^2$ is a KL function with $c$ given in \eqref{e:c&alpha}. Note that sum or quotient of two semi-algebraic functions is a semi-algebraic function, and indicator function of a semi-algebraic set (sets described as union or intersections of finitely many sets which can be expressed as lower level sets of polynomials) is also a semi-algebraic function. We note that this assumption is automatically satisfied when $f$ and $g$ are semi-algebraic functions, and $S$ is a semi-algebraic set. This, in particular, covers all the motivating examples we mentioned in the introduction.
\end{remark}

Next, we see that Algorithm~\ref{algo:main} converges in a linear rate when applied to the scale invariant sparse signal recovery problem and Rayleigh quotient optimization with spherical constraint, if the parameters  $\tau_n$ satisfy $\lim\inf_{n \to \infty} \tau_n =\overline{\tau} >0$.

\begin{proposition}[KL exponent $1/2$ \& linear convergence]
\label{p:linear}
Suppose that $\mathcal{H} =\mathbb{R}^N$ and one of the following holds:
\begin{enumerate}
\item\label{p:linear_rayleigh}
$f(x) =x^\top Ax +\iota_C(x)$, $g(x) =x^\top Bx$, and $S =\mathbb{R}^N$, where $A$ and $B$ are symmetric positive definite matrices and $C :=\{x\in \mathbb{R}^N: \|x\|=1\}$.
\item\label{p:linear_ssr}
$f(x) =\|x\|_1$, $g(x) =\|x\|_2$, and $S =\{x\in \mathbb{R}^N: Ax\leq b, Cx =d\}$, where $A\in \mathbb{R}^{M\times N}$, $b\in \mathbb{R}^M$, $C\in \mathbb{R}^{P\times N}$, and $d\in \mathbb{R}^P$.
\end{enumerate}
Then, for all $c\in \RP$, $h(x,y) =\frac{f(x)}{g(x)} +\iota_S(x) +c\|x-y\|^2$ satisfies the KL property with an exponent of $1/2$ at $(\overline{x},\overline{x})$ for all $\overline{x}\in \dom f$. Consequently, if $\lim\inf_{n \to \infty} \tau_n =\overline{\tau} >0$, then Algorithm~\ref{algo:main} exhibits linear convergence when applied to the above cases.
\end{proposition}
\begin{proof}
In view of \cite[Theorem~3.6]{LP18} and Theorem~\ref{t:cvgKL}, it suffices to show that $F :=f/g +\iota_S$ is a KL function with an exponent of $1/2$.

\ref{p:linear_rayleigh}: We see that $F(x) =\frac{x^\top Ax}{x^\top Bx} +\iota_C(x)$. For all $x\notin C$, $\partial_L F(x) =\varnothing$. For all $x\in C$, since $\partial_L \iota_C(x) =N_C(x) =\{\xi x: \xi\in \mathbb{R}\}$, it holds that
\begin{equation}\label{e:subdiffF}
\partial_L F(x) =\left\{\frac{2(x^\top Bx)Ax -2(x^\top Ax)Bx}{(x^\top Bx)^2} +\xi x: \xi\in \mathbb{R}\right\}.
\end{equation}
Let $\overline{x}\in \dom \partial_L F$. We must have $\overline{x}\in C$. Let $\varepsilon, \eta\in \left(0,1\right)$ and let $x$ be such that $\|x-\overline{x}\|\leq \varepsilon$ and $F(\overline{x}) <F(x) <F(\overline{x})+\eta$. Then $F(x) <+\infty$, and so $x\in C$. It follows from \eqref{e:subdiffF} that
\begin{align*}
\dist(0,\partial_L F(x)) &=\inf_{\xi\in \mathbb{R}} \left\|\frac{2(x^\top Bx)Ax -2(x^\top Ax)Bx}{(x^\top Bx)^2} +\xi x\right\| \\
&=\inf_{\xi\in \mathbb{R}} \left(\left\|\frac{2(x^\top Bx)Ax -2(x^\top Ax)Bx}{(x^\top Bx)^2}\right\|^2 +\xi^2\right)^{1/2} \\
&=\left\|\frac{2(x^\top Bx)Ax -2(x^\top Ax)Bx}{(x^\top Bx)^2}\right\| \\
&=\frac{2}{x^\top Bx}\left\|(A-F(\overline{x})B)x -(F(x)-F(\overline{x}))Bx\right\|,
\end{align*}
where the second equality follows from the fact that $x^\top(2(x^\top Bx)Ax -2(x^\top Ax)Bx) =0$ and $\|x\| =1$. Now, since $A-F(\overline{x})B$ is a symmetric matrix, there exists $c >0$ such that, for all $z\in \mathbb{R}^N$,
\begin{equation*}
\|(A-F(\overline{x})B)z\|^2\geq c(z^\top(A-F(\overline{x})B)z) =c(z^\top Bz)(F(z)-F(\overline{x})).
\end{equation*}
Therefore,
\begin{equation*}
\dist(0,\partial_L F(x))\geq 2(F(x)-F(\overline{x}))^{1/2}\left(\frac{\sqrt{c}}{\sqrt{x^\top Bx}} -(F(x)-F(\overline{x}))^{1/2}\frac{\|Bx\|}{x^\top Bx}\right).
\end{equation*}
Let $\lambda_{\max}$ and $\lambda_{\min}$ are the maximum and minimum eigenvalues of $B$, respectively. Then $\lambda_{\min} \leq x^\top Bx \leq \lambda_{\max}$ since $\|x\| =1$. By shrinking $\eta$ if necessary, we have
\begin{equation*}
(F(x)-F(\overline{x}))^{1/2}\frac{\|Bx\|}{x^\top Bx}\leq \eta^{1/2}\frac{\|Bx\|}{\lambda_{\min}}\leq \frac{\sqrt{c}}{2\sqrt{\lambda_{\max}}}.
\end{equation*}
We deduce that $\dist(0,\partial_L F(x))\geq \frac{\sqrt{c}}{\sqrt{\lambda_{\max}}}(F(x)-F(\overline{x}))^{1/2}$, and $F$ is thus a KL function with an exponent of $1/2$.

\ref{p:linear_ssr}: By a similar argument as in \cite[Theorem~4.4]{ZYP20}, $F$ is a KL function with an exponent of $1/2$.
\end{proof}

\section{Convergence to strong stationary points}
\label{s:strong}

In this section, we propose another algorithm which converges to a strong lifted stationary points of the fractional programming problem~\eqref{e:prob}. To do this, we now consider the case where Assumption~2 is replaced by the following stronger assumption.

\setcounter{assumption}{1}
\renewcommand{\theassumption}{\arabic{assumption}'}
\begin{assumption}
\label{a:g'}
$g(x) =\max\{g_i(x): 1\leq i\leq p\}$, where each $g_i$ is continuously differentiable on an open set containing $S$ and weakly convex on $S$ with modulus $\beta\in \RP$, and \eqref{a:bound} holds.
\end{assumption}

Recall that the \emph{$\varepsilon$-active set} for $g(x) =\max\{g_i(x): 1\leq i\leq p\}$ is defined by
\begin{equation*}
I_\varepsilon(x) =\{i \in \{1,\dots,p\}: g_i(x)\geq g(x)-\varepsilon\}.
\end{equation*}
We then propose an enhanced extrapolated proximal subgradient algorithm as follows.
\begin{tcolorbox}[
	left=0pt,right=0pt,top=0pt,bottom=0pt,
	colback=blue!10!white, colframe=blue!50!white,
  	boxrule=0.2pt,
  	breakable]
\begin{algorithm}[Enhanced extrapolated proximal subgradient algorithm]
\label{algo:max}
\step{}\label{step:initial'}
Choose $x_{-1} =x_0\in S\cap \dom f$ and set $n =0$. Let $\delta, \omega\in \RPP$, let $\zeta\in \RPP$ be such that $1-\sqrt{\beta}\zeta >0$, and let
\begin{equation*}
\overline{\mu}\in \left[0, \frac{\delta(1-\sqrt{\beta}\zeta)\sqrt{mM}}{2M}\right)
\quad\text{and}\quad
\overline{\kappa}\in \left[0,\sqrt{\frac{m\delta(1-\sqrt{\beta}\zeta)}{\ell M}-\frac{2m\overline{\mu}}{\ell\sqrt{mM}}} \, \right),
\end{equation*}
where $\ell$ is defined in Assumption~\ref{a:f}, $\beta$ is defined in Assumption~\ref{a:g'}, while $m$ and $M$ are given in \eqref{a:bound}. 

\step{}\label{step:main'}
Set $\theta_n =\frac{f(x_n)}{g(x_n)}$ and choose $\tau_n\in \mathbb{R}$ such that $0< \tau_n\leq 1/\max\{\beta\theta_n/(1-\zeta), \delta\}$. Let $u_n =x_n +\kappa_n(x_n-x_{n-1})$ with $\kappa_n\in [0,\overline{\kappa}]$ and $v_n =x_n +\mu_n(x_n-x_{n-1})$ with $\mu_n\in [0,\overline{\mu}\tau_n]$. For each $i_n\in I_\varepsilon(x_n)$, find
\begin{equation*}
w_n^{i_n} \in \argmin_{x\in S} \left(f^\mathfrak{n}(x) +f^\mathfrak{s}(u_n) +\scal{\nabla f^\mathfrak{s}(u_n)}{x-u_n} +\frac{1}{2\tau_n}\|x-v_n-\tau_n\theta_n\nabla g_{i_n}(x_n)\|^2 +\frac{\ell}{2}\|x-u_n\|^2\right).
\end{equation*}

\step{}\label{step:select}
Set $x_{n+1} :=w_n^{\hat{i}_n}$, where
\begin{equation*}
\hat{i}_n\in \argmin_{i_n\in I_\varepsilon(x_n)} \left(f(w_n^{i_n}) -\theta_ng(w_n^{i_n}) +\frac{1}{2}\left(\frac{1-\sqrt{\beta}\zeta}{\tau_n} -\frac{M\mu_n}{\sqrt{mM}\tau_n}\right)\|w_n^{i_n}-x_n\|^2\right).
\end{equation*}

\step{}If a termination criterion is not met, let $n =n+1$ and go to Step~\ref{step:main'}.
\end{algorithm}
\end{tcolorbox}

\noindent Before we proceed, we note that Step~\ref{step:select} in Algorithm~\ref{algo:max} is motivated by the recent work of Pang et al. \cite{PRA16} which proposes an enhanced version of the DC algorithm for solving DC programs that converges to a stronger notion of stationary points, namely, to d-stationary points. Similar to the work of Pang et al., in Step~\ref{step:main'}, we need to compute the proximal mapping of $f^\mathfrak{n}+\iota_S$ for  $|I_\varepsilon(x_n)|$ times (which is at most $p$). Although comparing to Algorithm~\ref{algo:main}, the computation cost in solving each subproblem may be higher, as we will see later, the algorithm converges to a strong lifted stationary point of \eqref{e:prob}.

\begin{theorem}
\label{t:max}
Let $(x_n)_{n\in \mathbb{N}}$ be the sequence generated by Algorithm~\ref{algo:max}. Suppose that Assumptions~\ref{a:f} and \ref{a:g'} hold, and that the set $\{x\in S: \frac{f(x)}{g(x)}\leq \frac{f(x_0)}{g(x_0)}\}$ is bounded. Then the following hold:
\begin{enumerate}
\item\label{t:max_decrease}
For all $n\in \mathbb{N}$, $x_n\in S\cap \dom f$ and
\begin{equation}
F_n :=\frac{f(x_n)}{g(x_n)} +\left(\frac{\ell\overline{\kappa}^2}{2m}+\frac{\overline{\mu}}{2\sqrt{mM}}\right)\|x_n-x_{n-1}\|^2
\end{equation}
is nonincreasing and convergent.
Consequently, the sequence $\left(\frac{f(x_n)}{g(x_n)}\right)_{n\in \mathbb{N}}$ is convergent.
\item\label{t:max_seq}
The sequence $(x_n)_{n\in \mathbb{N}}$ is bounded and asymptotically regular. In particular,
\begin{equation*}
\sum_{n=0}^{+\infty} \|x_{n+1}-x_n\|^2 <+\infty.
\end{equation*}
\item\label{t:max_crit}
If $\liminf_{n\to +\infty} \tau_n =\overline{\tau} >0$, then, for every cluster point $\overline{x}$ of $(x_n)_{n\in \mathbb{N}}$, it holds that $\overline{x}\in S\cap\dom f$, $\lim_{n\to +\infty} \frac{f(x_n)}{g(x_n)} =\frac{f(\overline{x})}{g(\overline{x})}$, and
\begin{equation}\label{eq:99}
\frac{f(\overline{x})}{g(\overline{x})}\bigcup_{i \in I_0(\overline{x})} \nabla g_i(\overline{x})\subseteq \partial_L(f+\iota_S)(\overline{x}).
\end{equation}
In addition, if $f$ is weakly convex on $S$, then $\overline{x}$ is a strong lifted stationary point of \eqref{e:prob}. 
\end{enumerate}
\end{theorem}
\begin{proof}
\ref{t:max_decrease}\&\ref{t:max_seq}: We first see that, for all $n\in \mathbb{N}$,
$x_n\in S\cap \dom f$,
and so
$g(x_n) >0 \text{~~and~~} \theta_n =\frac{f(x_n)}{g(x_n)}\geq 0$.

Next, for all $n\in \mathbb{N}$, $i_n\in I_\varepsilon(x_n)$, and $x\in S$,
\begin{align}\label{eq:use0}
f(w_n^{i_n}) &= f^\mathfrak{n}(w_n^{i_n})+f^\mathfrak{s}(w_n^{i_n}) \nonumber\\
&\leq f^\mathfrak{n}(w_n^{i_n}) + f^\mathfrak{s}(u_n) +\scal{\nabla f^\mathfrak{s}(u_n)}{w_n^{i_n}-u_n} +\frac{\ell}{2}\|w_n^{i_n}-u_n\|^2 \nonumber \\
&\leq f^\mathfrak{n}(x) + f^\mathfrak{s}(u_n) +\scal{\nabla f^\mathfrak{s}(u_n)}{x-u_n} +\frac{1}{2\tau_n}\|x-v_n-\tau_n\theta_n\nabla g_{i_n}(x_n)\|^2+\frac{\ell}{2}\|x-u_n\|^2 \nonumber\\
&\qquad -\frac{1}{2\tau_n}\|w_n^{i_n}-v_n-\tau_n\theta_n\nabla g_{i_n}(x_n)\|^2 \nonumber\\
&\leq f^\mathfrak{n}(x) +f^\mathfrak{s}(x) +\frac{1}{2\tau_n}\|x-v_n-\tau_n\theta_n\nabla g_{i_n}(x_n)\|^2 +\frac{\ell}{2}\|x-u_n\|^2 \nonumber\\
&\qquad -\frac{1}{2\tau_n}\|w_n^{i_n}-v_n-\tau_n\theta_n\nabla g_{i_n}(x_n)\|^2 \nonumber\\
&= f(x) +\frac{1}{2\tau_n}\|x-v_n\|^2 -\frac{1}{2\tau_n}\|w_n^{i_n}-v_n\|^2 +\theta_n\scal{\nabla g_{i_n}(x_n)}{w_n^{i_n}-x} +\frac{\ell}{2}\|x-u_n\|^2 \nonumber\\
&= f(x) +\frac{1}{2\tau_n}\|x-v_n\|^2 -\frac{1}{2\tau_n}\|x_n-v_n\|^2 -\frac{1}{2\tau_n}\|w_n^{i_n}-x_n\|^2 +\frac{\mu_n}{\tau_n}\scal{w_n^{i_n}-x_n}{x_n-x_{n-1}} \nonumber \\
&\qquad +\theta_n\scal{\nabla g_{i_n}(x_n)}{w_n^{i_n}-x_n} -\theta_n\scal{\nabla g_{i_n}(x_n)}{x-x_n} +\frac{\ell}{2}\|x-u_n\|^2,
\end{align}
where the first inequality is from the fact that $\nabla f^\mathfrak{s}$ is Lipschitz continuous with modulus $\ell$ (Lemma~\ref{l:descent}), the second inequality is from Step~\ref{step:main'} of Algorithm~\ref{algo:max}, the third inequality follows from the convexity of $f^\mathfrak{s}$, and the last equality uses the fact that $x_n-v_n =-\mu_n(x_n-x_{n-1})$.
For $\omega =\sqrt{m/M} >0$, one has from Young's inequality that
\begin{align}\label{eq:use1}
\scal{w_n^{i_n}-x_n}{x_n-x_{n-1}} &\leq \frac{1}{2\omega}\|w_n^{i_n}-x_n\|^2 +\frac{\omega}{2}\|x_n-x_{n-1}\|^2 \notag \\
&= \frac{M}{2\sqrt{mM}}\|w_n^{i_n}-x_n\|^2 +\frac{m}{2\sqrt{mM}}\|x_n-x_{n-1}\|^2.
\end{align}
It follows from Assumption~\ref{a:g'} that $g$ is regular and weakly convex with modulus $\beta$ on $S$. By Lemma~\ref{l:subgradine},
\begin{equation}\label{eq:use2}
\scal{\nabla g_{i_n}(x_n)}{w_n^{i_n}-x_n}\leq g_{i_n}(w_n^{i_n})-g_{i_n}(x_n) +\frac{\beta}{2}\|w_n^{i_n}-x_n\|^2.
\end{equation}
Combining inequalities \eqref{eq:use0}, \eqref{eq:use1} and \eqref{eq:use2}, and noting that $g_{i_n}(w_n^{i_n})\leq g(w_n^{i_n})$ by the definition of $g$ and that $\beta\theta_n\leq \sqrt{\beta}\zeta/\tau_n$ by the choice of $\tau_n$, one has
\begin{align*}
f(w_n^{i_n})
&\leq f(x) +\frac{1}{2\tau_n}\|x-v_n\|^2 -\frac{1}{2\tau_n}\|x_n-v_n\|^2 -\frac{1}{2}\left(\frac{1-\sqrt{\beta}\zeta}{\tau_n} -\frac{M\mu_n}{\sqrt{mM}\tau_n}\right)\|w_n^{i_n}-x_n\|^2 \\
&\qquad +\theta_n(g(w_n^{i_n})-g_{i_n}(x_n)) -\theta_n\scal{\nabla g_{i_n}(x_n)}{x-x_n} +\frac{\ell}{2}\|x-u_n\|^2 +\frac{m\mu_n}{2\sqrt{mM}\tau_n}\|x_n-x_{n-1}\|^2.
\end{align*}
Now, using the definition of $x_{n+1}$, we derive that, for all $n\in \mathbb{N}$, $i_n\in I_\varepsilon(x_n)$, and $x\in S$,
\begin{align}\label{e:fxfx+'}
&f(x_{n+1}) -\theta_ng(x_{n+1}) +\frac{1}{2}\left(\frac{1-\sqrt{\beta}\zeta}{\tau_n} -\frac{M\mu_n}{\sqrt{mM}\tau_n}\right)\|x_{n+1}-x_n\|^2 \notag \\
&\leq f(w_n^{i_n}) -\theta_ng(w_n^{i_n}) +\frac{1}{2}\left(\frac{1-\sqrt{\beta}\zeta}{\tau_n} -\frac{M\mu_n}{\sqrt{mM}\tau_n}\right)\|w_n^{i_n}-x_n\|^2 \notag \\
&\leq f(x) -\theta_ng_{i_n}(x_n) +\frac{1}{2\tau_n}\|x-v_n\|^2 -\frac{1}{2\tau_n}\|x_n-v_n\|^2 \notag \\
&\qquad -\theta_n\scal{\nabla g_{i_n}(x_n)}{x-x_n} +\frac{\ell}{2}\|x-u_n\|^2 +\frac{m\mu_n}{2\sqrt{mM}\tau_n}\|x_n-x_{n-1}\|^2.
\end{align}
Let $i_n\in I_0(x_n)\subseteq I_\varepsilon(x_n)$. Then $g_{i_n}(x_n) =g(x_n)$. Since $f(x_n) =\theta_ng(x_n)$ and $x_n-u_n =-\kappa_n(x_n-x_{n-1})$, letting $x =x_n$ in \eqref{e:fxfx+'} yields
\begin{equation*}
f(x_{n+1}) -\theta_ng(x_{n+1}) +\frac{1}{2}\left(\frac{1-\sqrt{\beta}\zeta}{\tau_n} -\frac{M\mu_n}{\sqrt{mM}\tau_n}\right)\|x_{n+1}-x_n\|^2 \leq \frac{1}{2}\left(\ell\kappa_n^2+\frac{m\mu_n}{\sqrt{mM}\tau_n}\right)\|x_n-x_{n-1}\|^2.
\end{equation*}
Dividing $g(x_{n+1})>0$ on both sides and recalling that $m \leq g(x_{n+1}) \leq M$, $\mu_n \leq \overline{\mu} \tau_n$, and $1/\tau_n \geq \delta$, we have that
{\small \begin{equation*}
\frac{f(x_{n+1})}{g(x_{n+1})} +\left(\frac{\delta(1-\sqrt{\beta}\zeta)}{2M} -\frac{\overline{\mu}}{2\sqrt{mM}}\right)\|x_{n+1}-x_n\|^2 \leq \frac{f(x_n)}{g(x_n)} +\left(\frac{\ell\kappa_n^2}{2m}+\frac{\overline{\mu}}{2\sqrt{mM}}\right)\|x_n-x_{n-1}\|^2.
\end{equation*}}
Proceeding as in the proof of Theorem~\ref{t:cvg}\ref{t:cvg_decrease}\&\ref{t:cvg_seq}, we obtain conclusions \ref{t:max_decrease} and \ref{t:max_seq} of this theorem.

\ref{t:max_crit}: In view of \ref{t:max_decrease}, we set
\begin{equation*}
\overline{\theta} :=\lim_{n\to +\infty} \theta_n =\lim_{n\to +\infty} \frac{f(x_n)}{g(x_n)}.
\end{equation*}
Let $\overline{x}$ be a cluster point of $(x_n)_{n\in \mathbb{N}}$ and let $(x_{k_n})_{n\in \mathbb{N}}$ be a subsequence convergent to $\overline{x}$. Then $\overline{x}\in S$ as well as $x_{k_n+1}\to \overline{x}$, $u_{k_n}\to \overline{x}$, and $v_{k_n}\to \overline{x}$ due to the asymptotic regularity of $(x_n)_{n\in \mathbb{N}}$. By the continuity of each $g_i$, there exists $n_0\in \mathbb{N}$ such that, for all $i\in \{1,\dots,p\}$ and all $n\geq n_0$, $g_i(x_{k_n})\geq g_i(\overline{x}) -\varepsilon/2$ and $g(\overline{x})\geq g(x_{k_n}) -\varepsilon/2$. It follows that, for all $n\geq n_0$,
$I_0(\overline{x})\subseteq  I_{\varepsilon}(x_{k_n})$.

Let $n\geq n_0$ and let $i \in I_0(\overline{x})\subseteq  I_{\varepsilon}(x_{k_n})$. We have from \eqref{e:fxfx+'} that, for all $x\in S$,
\begin{align}\label{e:fxkn+}
&f(x_{k_n+1}) -\theta_{k_n}g(x_{k_n+1}) +\frac{1}{2}\left(\frac{1-\sqrt{\beta}\zeta}{\tau_{k_n}} -\frac{M\mu_{k_n}}{\sqrt{mM}\tau_{k_n}}\right)\|x_{k_n+1}-x_{k_n}\|^2 \notag \\
&\leq f(x) -\theta_{k_n} g_i(x_{k_n}) +\frac{1}{2\tau_{k_n}}\|x-v_{k_n}\|^2 -\frac{1}{2\tau_{k_n}}\|x_{k_n}-v_{k_n}\|^2 \notag \\
&\qquad -\theta_{k_n}\scal{\nabla g_i(x_{k_n})}{x-x_{k_n}} +\frac{\ell}{2}\|x-u_{k_n}\|^2 +\frac{m\mu_{k_n}}{2\sqrt{mM}\tau_{k_n}}\|x_{k_n}-x_{k_n-1}\|^2.
\end{align}
It follows from the continuity of $g$, $g_i$, and $\nabla g_i$ that $g(x_{k_n+1})\to g(\overline{x})$, $g_i(x_{k_n})\to g_i(\overline{x}) =g(\overline{x})$ (as $i\in I_0(\overline{x}))$, and $\nabla g_i(x_{k_n})\to \nabla g_i(\overline{x})$. Letting $x =\overline{x}$ and $n \to +\infty$ in \eqref{e:fxkn+} and noting that $\overline{\tau} =\liminf_{k \to \infty}\tau_n >0$, we have
$\limsup_{n\to +\infty} f(x_{k_n+1}) \leq f(\overline{x})$.
Combining with the lower semicontinuity of $f$ gives $f(x_{k_n+1}) \to f(\overline{x})$ as $n \to +\infty$. Thus, $\theta_{k_n} \to \overline{\theta} =\frac{f(\overline{x})}{g(\overline{x})}$ as $n \to +\infty$.

Now, letting $n \to +\infty$ in \eqref{e:fxkn+}, we obtain that, for all $x \in S$,
\begin{equation*}
f(\overline{x}) \leq f(x) +\left(\frac{1}{2\overline{\tau}}+\frac{\ell}{2}\right)\|x-\overline{x}\|^2 +\frac{f(\overline{x})}{g(\overline{x})}\scal{\nabla g_{i}(\overline{x})}{\overline{x}-x}.
\end{equation*}
This shows that $\overline{x}$ minimizes the function $\varphi$ over $S$, where
\begin{equation*}
\varphi(x) :=f(x) +\left(\frac{1}{2\overline{\tau}}+\frac{\ell}{2}\right)\|x-\overline{x}\|^2 -\frac{f(\overline{x})}{g(\overline{x})}\scal{\nabla g_{i}(\overline{x})}{x}
\end{equation*}
In particular, one sees that, for all $i \in I_0(\overline{x})$,
$\frac{f(\overline{x})}{g(\overline{x})}  \  \nabla g_{i}(\overline{x}) \in \partial_L (f+\iota_S)(\overline{x})$.
So, $\overline{x}\in S\cap \dom f$ and
\begin{equation}\label{eq:strong_critical}
\bigcup_{i \in I_0(\overline{x})}\frac{f(\overline{x})}{g(\overline{x})}  \  \nabla g_{i}(\overline{x}) \subseteq \partial_L (f+\iota_S)(\overline{x}).
\end{equation}
By taking convex hull on both sides, we see that
\begin{equation*}
\frac{f(\overline{x})}{g(\overline{x})}\partial_L g(\overline{x})={\rm conv}\bigcup_{i \in I_0(\overline{x})}\frac{f(\overline{x})}{g(\overline{x})}  \  \nabla g_{i}(\overline{x}) \subseteq {\rm conv} \partial_L (f+\iota_S)(\overline{x}).
\end{equation*}
As $f$ is weakly convex on $S$, Lemma~\ref{l:weaklycvx}\ref{l:weaklycvx_cvx} implies that $\partial (f+\iota_S)(\overline{x})$ is convex. Thus, the conclusion follows.
\end{proof}

\begin{remark}[Absence of the boundedness condition]
As with Algorithm~\ref{algo:main} and Theorem~\ref{t:cvg}, in the case where \eqref{a:bound} fails, if we set $\overline{\mu} =\overline{\kappa} =0$ in Step~\ref{step:initial'} and let
\begin{equation*}
\hat{i}_n\in \argmin_{i_n\in I_\varepsilon(x_n)} \left(f(w_n^{i_n}) -\theta_ng(w_n^{i_n}) +\frac{1-\sqrt{\beta}\zeta}{2\tau_n}\|w_n^{i_n}-x_n\|^2\right)
\end{equation*}
in Step~\ref{step:select} of Algorithm~\ref{algo:max}, then Theorem~\ref{t:max} still holds with $F_n =\frac{f(x_n)}{g(x_n)}$.
\end{remark}

\begin{remark}[Discussion of the results]
\begin{enumerate}
\item
Firstly, a close inspection of the proof and noting that, for all $\eta<\varepsilon$, one has  for all large $n$, $I_\eta(\overline{x})\subseteq  I_{\varepsilon}(x_{k_n})$.
So, \eqref{eq:99} in the conclusion of Theorem~\ref{t:max}\ref{t:max_crit} indeed can be strengthened as: for all $\eta<\varepsilon$,
\begin{equation*}
\frac{f(\overline{x})}{g(\overline{x})}\bigcup_{i \in I_\eta(\overline{x})} \nabla g_i(\overline{x})\subseteq \partial_L(f+\iota_S)(\overline{x}).
\end{equation*}
\item
Secondly, following the same method of proof used in Theorem \ref{t:cvgKL}, one can establish the global convergence of Algorithm~\ref{algo:max} under the KL assumptions in Theorem \ref{t:cvgKL} and also the additional assumption that $I_0(\overline{x}) =\{i\in\{1,\dots,p\}: g_i(\overline{x})=g(\overline{x})\}$ is a singleton for all $\overline{x} \in \Omega$, where $\Omega$ is the set of cluster points of $(x_n)_{n\in \mathbb{N}}$. Another sufficient condition ensuring the global convergence would be any point $\overline{x} \in\Omega$ is isolated. For brevity purpose, we omit the proof here. Unfortunately, these conditions are rather restrictive for the setting of Algorithm~\ref{algo:max}. It would be interesting to see how one can obtain further weaker conditions ensuring the global convergence of Algorithm~\ref{algo:max}. This would be an interesting open question and will be examined later.
\end{enumerate}
\end{remark}

\section{Numerical examples}
\label{s:numerical}

In the section, we illustrate our proposed algorithms via numerical examples. We first start with an explicit analytic example and use it to demonstrate the different behavior of Algorithm~\ref{algo:main} and Algorithm~\ref{algo:max} as well as the effect of the extrapolations. Then, we examine the performance of the algorithm for the scale invariant sparse signal reconstruction model. All the numerical tests were conducted on a computer with a 2.8 GHz
Intel Core i7 and 8 GB RAM, equipped with MATLAB R2015a.

\subsection{An analytical example}

Consider the analytical example discussed in Example~\ref{e:ex}
\begin{equation}\label{e:analytic}
\min_{x \in [-1,1]} \frac{x^2+1}{|x|+1}. \tag{EP$_1$}
\end{equation}
In this case, $g(x)=|x|+1$ is convex, and so, $\beta=0$. Also, for all $x \in [-1,1]$, $m \leq g(x) \leq M$, where $m=1$ and $M=2$. The numerator $f(x)=f^{\mathfrak{s}}(x)=x^2+1$ is a convex and continuously differentiable function whose gradient is Lipschitz continuous with modulus $\ell=2$.

{\bf Algorithm~\ref{algo:main} vs. Algorithm~\ref{algo:max}.} Let $\delta=\frac{\ell M}{m}=4$ and $\tau_n=\frac{1}{\delta}=\frac{1}{4}$ for all $n$. Set $\overline{\mu}=0$ and let $\overline{\kappa} \in (0,1)$ and $\kappa_n \in [0,\overline{\kappa}]$. We now compare the behavior of Algorithm~\ref{algo:main} and Algorithm~\ref{algo:max} for \eqref{e:analytic}:

Firstly, it can be directly verified that $g_n= {\rm sign}(x_n) \in \partial g(x_n)$ and that $f^\mathfrak{s}(u_n) +\langle \nabla f^\mathfrak{s}(u_n),x-u_n\rangle +\frac{\ell}{2}\|x-u_n\|^2 =x^2+1$. In this case, Algorithm~\ref{algo:main} reduces to
\begin{equation*}
x_{n+1}= {\rm P}_{[-1,1]} \left(\frac{2}{3}\left[x_n+\frac{1}{4} \frac{x_n^2+1}{|x_n|+1} {\rm sign}(x_n)\right]\right).
\end{equation*}
If one chooses as initial point $x_0=0$, then $x_n=0$ for all $n$, and so, $(x_n)_{n \in \mathbb{N}}$ converges to a lifted stationary point (but not a strong lifted stationary point).

If one chooses as initial point $x_0>0$, then, by induction, it is easy to see that $x_n>0$ and so, $x_n \in (0,1]$. This implies that
\begin{equation*}
x_{n+1}= {\rm P}_{[-1,1]} \left(\frac{2}{3}\left[x_n+ \frac{x_n^2+1}{4(x_n+1)} \right]\right)= \frac{2}{3}\left[x_n+ \frac{x_n^2+1}{4(x_n+1)} \right],
\end{equation*}
where  the last equality is from the fact that $x_n+ \frac{x_n^2+1}{4(x_n+1)} \in [0,\frac{3}{2}]$ for all $x_n \in (0,1]$. Here, $P_{[-1,1]}$ denotes the Euclidean projection onto the set $[-1,1]$. Thus, $x_n \rightarrow \sqrt{2}-1$ which is a lifted stationary point.

Similarly, if one chooses as initial point $x_0<0$, then, $x_n \rightarrow 1-\sqrt{2}$ which is also a lifted stationary point.

Next, we analyze the behavior of Algorithm~\ref{algo:max}. Recall that $\delta=\frac{\ell M}{m}=4$, $\tau_n=\frac{1}{\delta}=\frac{1}{4}$, $\overline{\mu}=0$, $\kappa_n \in [0,\overline{\kappa}]$ with $\overline{\kappa} \in (0,1)$. Let $\varepsilon=2$. Note that $g(x)=\max\{x+1,-x+1\}$. Then $I_{\varepsilon}(x_n)=\{1,2\}$, and so,
\begin{equation*}
w_{n}^1= {\rm P}_{[-1,1]} \left(\frac{2}{3}\left[x_n+\frac{1}{4} \frac{x_n^2+1}{|x_n|+1}  \right]\right)
\text{~and~}
w_{n}^2= {\rm P}_{[-1,1]} \left(\frac{2}{3}\left[x_n-\frac{1}{4} \frac{x_n^2+1}{|x_n|+1}  \right]\right).
\end{equation*}
In Algorithm~\ref{algo:max}, we set $x_{n+1} :=w_n^{\hat{i}_n}$, where
\begin{equation*}\hat{i}_n\in \argmin_{i \in \{1,2\}} \left((w_n^{i})^2+1 - \frac{x_n^2+1}{|x_n|+1} (|w_n^{i}|+1) +2(w_n^{i}-x_n)^2\right).
\end{equation*}
For the proceeding step for updating $x_{n+1}$, if the values happens to be the same in the above {\rm argmin} operations, we choose $\hat{i_n}$ to be the smallest index. By randomly generated the initial guess $x_0$, we observe that Algorithm~\ref{algo:max} generates a sequence $(x_n)_{n \in \mathbb{N}}$ such that $x_n \rightarrow \sqrt{2}-1$ if $x_0 \geq 0$ and $x_n \rightarrow 1-\sqrt{2}$ if $x_0<0.$ Figure~\ref{fig:0} depicts the trajectory $x_n$ of Algorithm~\ref{algo:max} with three initial points: $x_0=0, -1, 1$. Interestingly, we note that, in the case where $x_0=0$, Algorithm~\ref{algo:max} converges to a strong lifted stationary point $\sqrt{2}-1$ while Algorithm~\ref{algo:main} converges to a lifted stationary point $0$, which is not a strong lifted stationary point.
\begin{figure}[!htb]
\centering
\includegraphics[scale=0.32]{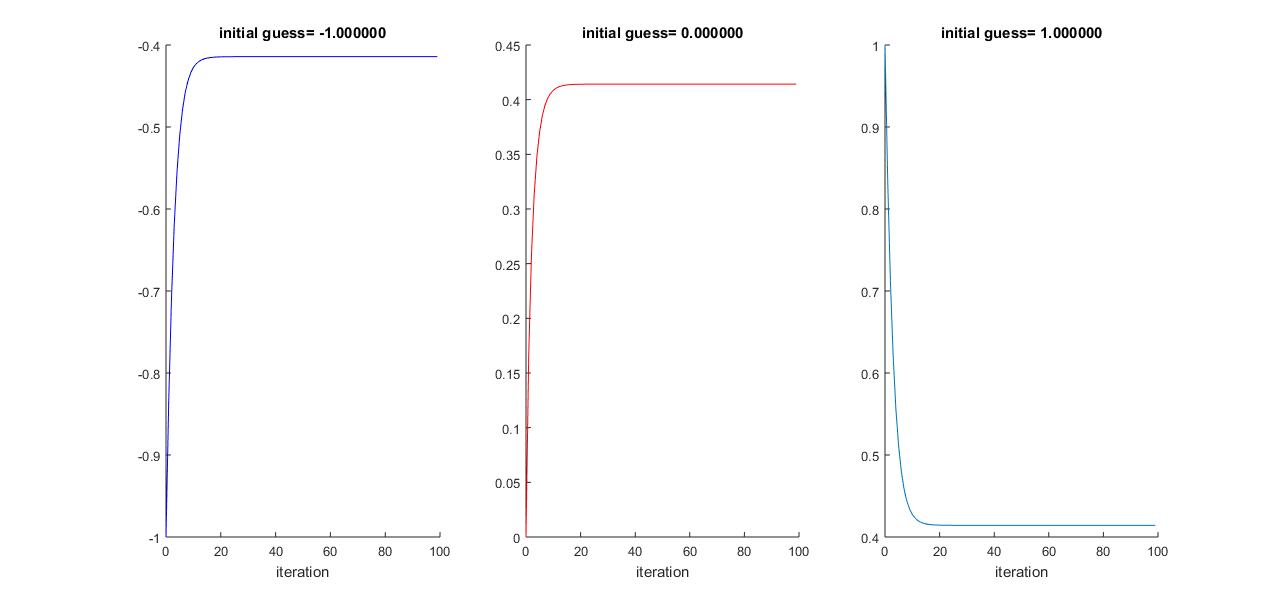}
\vspace{-0.6cm}
\caption{Trajectory of Algorithm~\ref{algo:max} with different initial guess $x_0$ for \eqref{e:analytic}}\label{fig:0}
\end{figure}

{\bf Effect of the extrapolation parameter.} We now illustrate the behavior of Algorithm~\ref{algo:main} by varying the extrapolation parameters. To do this, let $\beta=0$, $\delta=\frac{\ell M}{m}=4$, $\tau_n=\frac{1}{\delta}=\frac{1}{4}$ for all $n$. Fix any $\overline{\kappa} \in (0,1)$ and $\kappa_n \in [0,\overline{\kappa}]$. Let $\alpha \in [0,1)$. Set $g_n={\rm sign}(x_n) \in \partial_L g(x_n)$, $\overline{\mu} =\frac{\alpha  \delta \sqrt{mM}}{2M} =\sqrt{2} \alpha$, and $\mu_n=\frac{\sqrt{2}}{4}\alpha \frac{\nu_{n-1}-1}{\nu_n}$, where
\begin{equation*}
\nu_{-1}=\nu_0=1, \text{~and~} \nu_{n+1}=\frac{1+\sqrt{1+4 \nu_n^2}}{2},
\end{equation*}
and reset $\nu_{n-1}=\nu_n=1$ when $n=n_0,2n_0,3n_0,\dots$ for the integer $n_0=50$. In this case, direct verification shows that $\sup_{n} \nu_n \leq 1$, and hence $\mu_n \leq \frac{\sqrt{2}}{4} \alpha =\overline{\mu}\tau_n$. Starting with the initialization $x_0=1$, we then run Algorithm~\ref{algo:main} with different $\alpha \in[0,1)$. Figure~\ref{fig:trajectory} depicts the distance, in the log scale, between the iterates $x_n$ and the solution $x^*=\sqrt{2}-1$ for $\alpha \in \{0,0.5,0.7,0.99\}$, where the case $\alpha=0$ indeed corresponds to the un-extrapolated cases. As one can see from Figure~\ref{fig:trajectory},  as $\alpha$ increases and approaches $1$, the algorithm tends to converge faster.
\begin{figure}[!htb]
\centering
\includegraphics[scale=0.26]{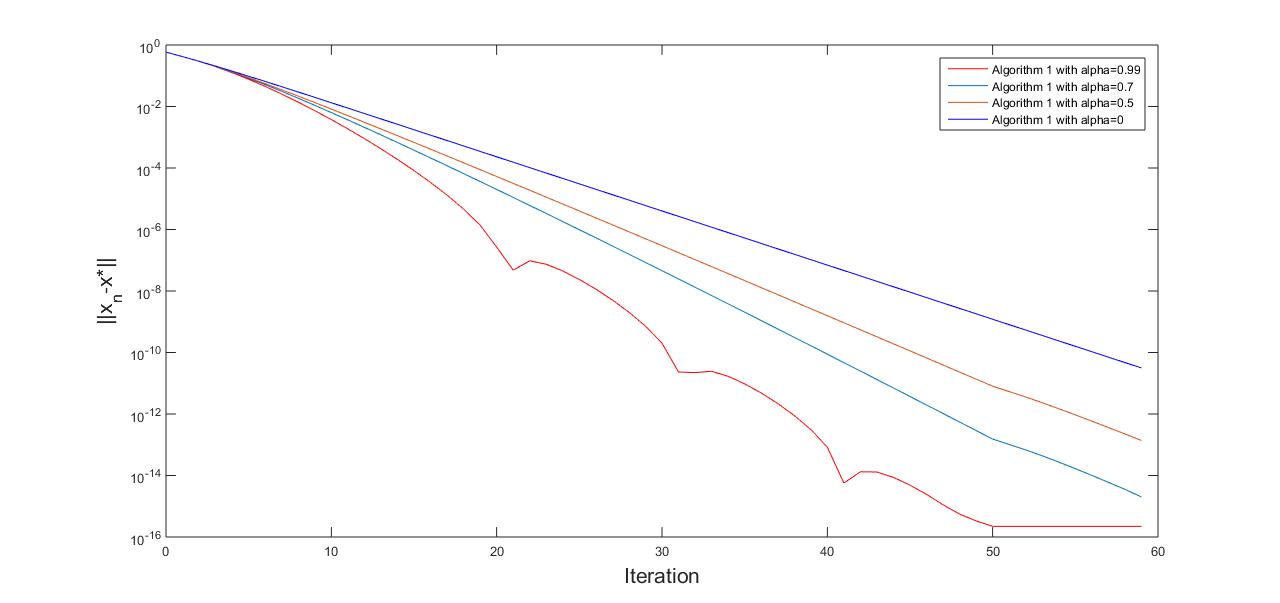}
\caption{Distance to the solution vs iterations in solving \eqref{e:analytic}}\label{fig:trajectory}
\vspace{-0.6cm}
\end{figure}

\subsection{Scale invariant sparse signal recovery problem}

As another illustration, we examine the following scale invariant sparse signal recovery problem discussed in the motivating example
\begin{equation}\label{e:ssr}
\min_{x\in \mathbb{R}^N} \frac{\|x\|_1}{\|x\|_2} \quad\text{s.t.}\quad Ax = b,\  {\rm lb}_i \leq x_i \leq {\rm ub}_i,\ i=1,\dots,N,
\tag{EP$_2$}
\end{equation}
where ${\rm lb}_i$ and ${\rm ub}_i$ are the lower bound and upper bound for the variables $x_i$, $i=1,\dots,N$. We follow \cite{RWDL19} and generate the matrix $A$ via the so-called oversampled discrete cosine transform (DCT), that is,
$A=[a_1, a_2,\dots,a_N] \in \mathbb{R}^{P \times N}$ where
\begin{equation*}
a_j=\frac{1}{\sqrt{P}} \cos\left(\frac{2\pi w \, j}{F}\right), \ j=1,\dots,N.
\end{equation*}
where $w$ is a random vector uniformly distributed in $[0,1]^P$
 and $F$ is a positive number which gives a measure on how coherent the matrix is. The ground truth $x^g \in \mathbb{R}^N$ is simulated as an $s$-sparse signal where $s$ is the total number of nonzero entries. The support of $x^g$ is a random index set, and the values of nonzero elements follow a Gaussian normal distribution. Then the ground-truth is normalized to have maximum
 magnitude as $1$ so that we can examine the performance within the $[-1,1]^N$ box constraint. Then, we generate $b=Ax^g$, and set ${\rm lb}_i=-1$ and ${\rm ub}_i=1$. Specifically, in our experiment, following \cite{RWDL19}, we consider the above matrix $A$ of size
 $(P,N)=(64,1024)$, $F=10$ and the ground-truth sparse vector has $12$ nonzero elements.

We use two methods for solving this  scale invariant sparse signal recovery problem: our proposed extrapolated proximal subgradient method (e-PSG)  and the alternating direction of method of multipliers (ADMM) proposed in \cite{RWDL19}. It was shown
in \cite{RWDL19} that the ADMM method works very efficiently although the theoretical justification of the convergence of this method is still lacking.
\begin{itemize}
 \item ADMM method: We first solve the $L_1$-optimization problem which results when replacing the objective of \eqref{e:ssr} by $\|x\|_1:=\sum_{i=1}^N |x_i|$. This is done by using the commercial
 software Gurobi and produces a solution $x_0$ for the $L_1$-optimization problem. Following \cite{RWDL19}, we use $x_0$ as an initialization and use the ADMM method proposed therein. We terminate
 the algorithm when the relative error $\frac{\|x_{n+1}-x_n\|}{\max\{\|x_n\|,1\}}$ is smaller than $10^{-9}$.
 \item Algorithm~\ref{algo:main} (e-PSG method): Similar to the ADMM method, we also use the solution of  the $L_1$-optimization problem as the initial point. We choose $f^s \equiv 0$ (and so, $\ell=0$),  $\kappa_n=0$. As $g(x)=\|x\|_2$ is convex,  $\beta=0$. Moreover, for all $x$ feasible for \eqref{e:ssr},
 $m \leq g(x) \leq M$ where $M=\sqrt{N}$ and $m$ is a positive number computed as the Euclidean norm of the least norm solution of $Ax=b$ via the Matlab code \verb+m = norm(pinv(A)*b)+.
 Let $\alpha=0.99$ and set $\mu_n= \frac{\alpha \sqrt{\frac{m}{M}}}{2} \frac{\nu_{n-1}-1}{\nu_n}$, where
\begin{equation*}
\nu_{-1}=\nu_0=1, \text{~and~} \nu_{n+1}=\frac{1+\sqrt{1+4 \nu_n^2}}{2},
\end{equation*}
and reset $\nu_{n-1}=\nu_n=1$ when $n=n_0,2n_0,3n_0,\dots$ for the integer $n_0=50$. For any $\delta>0$,
let $\tau_n=\frac{1}{\delta}$ and $\overline{\mu}= \frac{\alpha \delta \sqrt{\frac{m}{M}}}{2}< \frac{\delta \sqrt{\frac{m}{M}}}{2}$. It can be verified that $\mu_n \leq \frac{\alpha \sqrt{\frac{m}{M}}}{2} = \overline{\mu} \tau_n$, and so, the requirements of the parameters in Algorithm~\ref{algo:main}
are satisfied.  We use the same termination criterion as for the ADMM method. For the subproblem arising in Step~\ref{step:main} of Algorithm~\ref{algo:main}, we reformulate the problem as an equivalent quadratic program with linear constraints, and
 solve it using the
 software Gurobi.
\end{itemize}

We run the ADMM and the e-PSG method (Algorithm~\ref{algo:main}) for 50 trials. The following table summarizes the output of the two methods by listing the average number of
\begin{itemize}
\item   sparsity level of the initial guess: the number of entries of the initialization (the solution for $L_1$-optimization problem) with value larger than $10^{-6}$;
 \item sparsity level of the solution: the number of entries of the computed solution with value larger than $10^{-6}$;
 \item error with respect to the ground truth: the Euclidean norm of the difference of the computed solution and the ground truth $x^g$;
 \item the objective value of the computed solution;
 \item CPU time measured in seconds.
\end{itemize}
From Table~\ref{tab:computation}, one can see that e-PSG method is competitive with the ADMM method in terms of sparsity level and the CPU time used, and produces a solution with slightly better quality in terms
of the final objective value and the error with respect to the ground truth.  As plotted in Figure~\ref{fig:1}, one
can see that ADMM uses around 2000 iterations to reach the desired relative error tolerance, and has  sharp oscillating phenomenon in terms of the objective value (this has also been observed in \cite{RWDL19}, and the authors of \cite{RWDL19} believed that this is one of the major obstacles in establishing the convergence of the ADMM method); while the proposed e-PSG method quickly approaches
the desired error tolerance. On the other hand, it should be noted that the subproblems in the ADMM method have closed form solutions while the subproblems in the e-PSG method are reformulated as quadratic programming problems  with linear constraints and solved via the software Gurobi\footnote{One possible way to improve the CPU time in using e-PSG is to solve the subproblem via alternating direction method of multiplier method directly. We leave this as a future study.}.

\begin{table}[!htb]
\centering
\begin{tabular}{|l|c|c|c|c|c|} \hline
 & \multicolumn{2}{c|}{sparsity level} & error w.r.t & objective value of  & \multirow{2}{*}{CPU time} \\ \cline{2-3}
 & initial guess & computed solution & the ground truth & the computed solution   & \\ \hline
 ADMM & 64 & 12 & 6.948329e-06 & 2.724348 & 1.970365 \\ \hline
 e-PSG & 64 & 12  & 4.539185e-10 & 2.724326 & 2.375557 \\ \hline
\end{tabular}
\vspace{-0.2cm}
\caption{Computation results for \eqref{e:ssr}}\label{tab:computation}
\end{table}

\begin{figure}[!htb]
\centering
\includegraphics[scale=0.33]{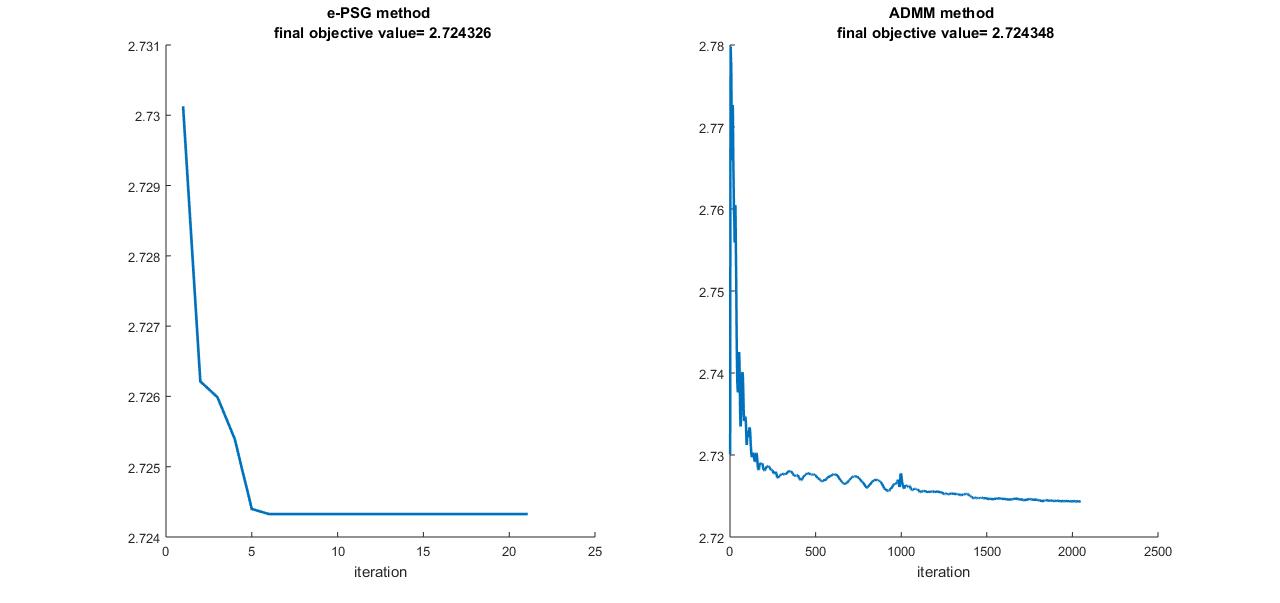}
\caption{Objective values vs. iterations in solving \eqref{e:ssr}}\label{fig:1}
\vspace{-0.6cm}
\end{figure}

\section{Conclusions}

We have proposed proximal subgradient algorithms with extrapolations for solving fractional optimization model where both the numerator and denominator can be nonsmooth and nonconvex. We have shown that the sequence of iterates generated by the algorithm is bounded and any of its limit points is a stationary point of the model problem. We have also established the global convergence of the sequence by further assuming the KL property for a suitable merit function by providing a unified analysis framework of descent methods. Finally, in the case where the denominator is the maximum of finitely many continuously differentiable weakly convex functions, we have also proposed an enhanced proximal subgradient algorithm with extrapolations, and showed that this enhanced algorithm converges to a stronger notion of stationary points of the model problem.

Our results in this paper point out the following interesting open questions and future work:
(1) For the enhanced proximal subgradient algorithm with extrapolations (Algorithm~\ref{algo:max}), is it possible to extend the case from
$g(x) =\max_{1 \leq i \leq p}\{g_i(x)\}$ to $g(x) =\max_{t \in T}\{g_t(x)\}$ where $T$ is a (possibly) infinite set? (2) In Algorithm~\ref{algo:max}, as one needs to solve the subproblem $|I_{\varepsilon}(x_n)|$ times, this can be time consuming when the dimension is high. Is it possible to incorporate randomize techniques to save the computational cost and establish the convergence in probability sense? (3) How to obtain the global convergence of the full sequence of Algorithm~\ref{algo:max} under weaker and
reasonable assumptions is also an important topic to be examined. Finally, further numerical implementations of our algorithms and comparisons with other competitive methods are left as future research.

\noindent {\bf Acknowledgement:} The authors would like to thank Dr. Yifei Lou for kindly sharing the MATLAB code for the ADMM method used in \cite{RWDL19}.

\end{document}